\documentclass{amsproc}

\usepackage{latexsym}
\usepackage{amsmath}
\usepackage{amssymb}
\usepackage{amsfonts}
\usepackage{verbatim}
\usepackage[latin1]{inputenc}
\usepackage{mathrsfs}
\usepackage{amsfonts}
\usepackage{graphicx}
\usepackage{colortbl,dcolumn}
\usepackage{amsmath}
\usepackage{psfrag}
\usepackage{booktabs}
\usepackage{lipsum}
\usepackage{amsfonts}
\usepackage{graphicx}
\usepackage{epstopdf}
\usepackage{algorithmic}
\ifpdf
  \DeclareGraphicsExtensions{.eps,.pdf,.png,.jpg}
\else
  \DeclareGraphicsExtensions{.eps}
\fi

\newtheorem{tm}{Theorem}[section]
\newtheorem{rk}{Remark}[section]

\newtheorem{prop}{Proposition}[section]
\newtheorem{lm}{Lemma}[section]
\newtheorem{cor}{Corollary}[section]

\newcommand{\OOO} {\mathscr O}
\newcommand{\E}{\mathbb E}
\newcommand{\PP}{\mathbb P}
\newcommand{\N}{\mathbb N}
\newcommand{\R}{\mathbb R}

\newcommand{\bs}{\mathbf s}

\newcommand{\LL}{\mathcal L}
\newcommand{\OO}{\mathcal O}
\newcommand{\HH}{\mathbb H}

\newcommand{\FFF}{\mathscr F}
\newcommand{\<}{\langle}
\renewcommand{\>}{\rangle}

\ifpdf
  \DeclareGraphicsExtensions{.eps,.pdf,.png,.jpg}
\else
  \DeclareGraphicsExtensions{.eps}
\fi

\begin{document}

\title[Strong Convergence Rate of Splitting approximation   for SACE]
{Strong convergence rates of semi-discrete splitting approximations for stochastic Allen--Cahn equation}

\author{Charles-Edouard  Br\'ehier}
\address{Univ Lyon, CNRS, Universit\'e Claude Bernard Lyon 1, UMR5208, Institut Camille Jordan, F-69622 Villeurbanne, France}
\curraddr{}
\email{brehier@math.univ-lyon1.fr}
\thanks{}
\author{Jianbo Cui}
\address{1. LSEC, ICMSEC, Academy of Mathematics and Systems Science, Chinese Academy of Sciences, Beijing, 100190, China 
2. School of Mathematical Science, University of Chinese Academy of Sciences, Beijing, 100049, China
}
\curraddr{}
\email{jianbocui@lsec.cc.ac.cn(Corresponding author)}
\thanks{}
\author{Jialin Hong}
\address{1. LSEC, ICMSEC, Academy of Mathematics and Systems Science, Chinese Academy of Sciences, Beijing, 100190, China 
2. School of Mathematical Science, University of Chinese Academy of Sciences, Beijing, 100049, China
}
\curraddr{}
\email{hjl@lsec.cc.ac.cn}
\thanks{}




\date{\today}

\dedicatory{}

\begin{abstract}
This article analyzes an explicit temporal splitting  numerical scheme for the stochastic Allen-Cahn equation
driven by additive noise, in a bounded spatial domain with smooth boundary in dimension $d\le 3$. The splitting strategy is combined with an exponential Euler scheme of an auxiliary problem.

When $d=1$ and the driving noise is a space-time white noise,
we first show some a priori estimates of this splitting scheme. Using the monotonicity of the drift nonlinearity, we then prove that 
under very mild assumptions on the initial data,
this scheme achieves
the optimal strong convergence rate $\OO(\delta t^{\frac 14})$. 
When $d\le 3$ and the driving noise possesses some regularity in space, we  study exponential integrability properties of the exact  and numerical solutions. Finally, in dimension $d=1$, these properties are used to prove that the splitting scheme has a strong convergence rate $\OO(\delta t)$.
\end{abstract}

\maketitle


\section{Introduction}

The stochastic Allen-Cahn equation driven by an additional noise term models the effect of thermal perturbations, and plays an important role in the phase
theory and the simulations of rare events in infinite dimensional stochastic systems (see e.g. \cite{Fun95, KORV07,VW12}). 
 
In this article, 
we mainly focus on deriving the optimal strong convergence rates of  temporal splitting schemes for  
the stochastic Allen-Cahn equation driven by 
Wiener processes, including the cylindrical Wiener process and some more regular Wiener processes,
under homogenous Dirichlet boundary conditions:  
\begin{align}\label{ac}
dX(t)=AX(t)+F(X(t))dt+dW^Q(t),\; t\in (0,T],\;\; X(0)=X_0,
\end{align}
where $F(x)=x-x^3$,
$(W^Q(t))_{t\in [0,T]}$ is a generalized  Wiener process  on a filtered probability space $(\Omega, \FFF, (\FFF(t))_{t\in[0,T]},\PP)$ and $\OO\in \R^d, d\le 3$ is a bounded spatial domain with smooth boundary $\partial \OO$.

Strong convergence of numerical approximations for Stochastic Partial Differential 
Equations(SPDEs) with  globally Lipschitz continuous coefficients has been extensively studied in the 
last twenty years (see e.g. \cite{ACLW16,
Gyo98,JKW11,Kru14a,Pri02}).
For SPDEs with non-Lipschitz coefficients, 
there only exist a few results about the strong convergence rates of numerical schemes (see e.g. \cite{ BGJK17,BJ16, CHL16b, CHLZ17,FLZ17}). The strong convergence rates of numerical schemes, especially the temporal discretization, is far from being understood and it is still an open problem to derive  general strong convergence rates of numerical schemes for SPDEs with non-globally Lipschitz  coefficients.

For the discretization of equations such as the stochastic Allen-Cahn equation, the main difficulty is the polynomial growth of the non-globally Lipschitz continuous coefficient $F$. It is very delicate and necessary to design efficient numerical schemes for stochastic equations with this type of nonlinearities.  The authors in~\cite{KLL18} 
study a Euler type split-step scheme and the backward Euler scheme, and show 
that these schemes converge strongly with a rate 
$\OOO(\delta t^{\frac 12})$ for Eq. \eqref{ac} with $d\le 3$, driven by some $Q$-Wiener processes. 
 We refer to~\cite{FLZ17} for the analysis of finite element methods applied to stochastic Allen-Cahn equations with multiplicative noise.
For Eq.~\eqref{ac}, with $d= 1$ driven by a space-time white noise, first, the authors in~\cite{BJ16} obtain the strong 
convergence rate results for a nonlinearity-truncated 
Euler-type scheme. Similar strong convergence results are then obtained in~\cite{BGJK17} for a nonlinearity-truncated fully discrete scheme.
A backward Euler-Spectral Galerkin method has been considered in~\cite{LQ18} by using stochastic calculus in martingale type 2 Banach spaces. The authors  in \cite{BG18} propose some splitting schemes and prove the proposed schemes are strongly convergent without strong convergence rates.

In this work, we give a systemic analysis of the properties of a splitting scheme and its strong convergence rates for approximating  Eq. \eqref{ac} with $d\le 3$ driven by different kinds of noise.
We first introduce the  splitting scheme
with a time-step size $\delta t>0$, defined by:
\begin{align}\label{spl}
Y_n&=\Phi_{\delta t}(X_n),\\\nonumber
X_{n+1}&=S_{\delta t}Y_n+\int_{t_n}^{t_{n+1}}S(t_{n+1}-s)dW^Q(s),
\end{align}
where  $\Phi_{\delta t}(z)=\frac z{\sqrt{z^2+(1-z^2)e^{-2\delta t}}}$ is the phase flow of 
$dX=F(X(t))dt, t\in [0,\delta t], X_0=z$, and
$S_{\delta t}=S(\delta t)=e^{A\delta t}$.
This type of splitting scheme, in a stochastic context, has been first proposed in~\cite{BG18}, and it is convenient for  practical implementations since it is explicit and strongly convergent without a taming or truncation strategy.  Note that an exponential Euler scheme is used in the second step of the splitting strategy.

In this article, we first derive the optimal strong convergence rate of the splitting scheme in the case of space-white time noise, using a variational approach.
This gives a positive answer to the question asked in~\cite{BG18}, concerning the strong convergence rate of splitting schemes for the stochastic Allen-Cahn equation.
We would like to mention that these splitting-up based methods have many applications on approximating SPDEs with
the Lipschitz nonlinearity, and are also used for approximating  SPDEs with non-Lipschitz or non-monotone nonlinearities (see e.g. \cite{CH17,CHLZ17,Dor12, GK03}).

In order to analyze the strong convergence rate of this splitting method for different types of noise, different approaches  are required. In the case of space-time white noise, there are three main steps to derive the strong convergence rate. Following~\cite{BG18}, the first step is constructing an auxiliary problem, with a modified nonlinearity $\Psi_{\delta t}$ instead of $F$,
 such that the splitting scheme can be viewed as   a standard exponential Euler scheme applied to the auxiliary problem.
Even though the exponential Euler  applied to the original equation may be divergent, the  solutions of the numerical scheme and of the auxiliary problem are proved to be bounded in $L^p(\Omega; L^q)$, for all finite $p,q$.  Thus no taming or truncation strategy is required to ensure the  boundedness of numerical solutions.
The second step is based on the monotonicity properties of the nonlinearities $F$ and $\Psi_{\delta t}$, appearing in the exact and
auxiliary problems respectively. In addition, since the noise is additive and an exponential Euler scheme is used with no discretization of the stochastic convolution, one is lead to study some PDEs with random coefficients. 
The last step  consists in applying  properties of the stochastic convolution and stochastic calculus results in martingale  
type 2 Banach spaces, to deduce the optimal strong convergence rate $\OOO(\delta t^{\frac 14})$ in $L^p(\Omega; C(0,T; L^q))$, $p\ge q=2m$, $m\in \N^+$, i.e.,
\begin{align*}
\Big\|\sup_{t\in [0,T]}\|X^N(t)-X(t)\|_{L^{q}}\Big\|_{L^p(\Omega)}
\le C(X_0,T,p,q)\delta t^{\frac 14}.
\end{align*}
This variational approach can also be used to obtain the strong convergence rate $\OOO(\delta t^{\frac 12})$,
in the case of more regular $Q$-Wiener processes,  in dimension $d\le 3$.

In the case of $\HH^1$-valued $Q$-Wiener processes, in dimension $d=1$, we get higher strong convergence rates of this splitting method,  thanks to exponential integrability properties of the exact and numerical solutions.
To the best of our knowledge,this is the first result with strong convergence order $1$ about the temporal numerical  schemes approximating the stochastic Allen--Cahn equation.
For similar approaches to derive the strong convergence rates of numerical schemes, we refer to 
\cite{CH17, HJW13, JP16} and the references therein.
We first study stability and exponential integrability properties of the exact solution, in dimensions $d\le 3$, and obtain results of their own interest beyond analysis of numerical schemes. Then,  in dimension $d=1$, a new auxiliary processes $Z^N$ is constructed, and some a priori estimate and exponential integrability properties of $Z^N$ are studied. 
We then  prove that the scheme, in this context, has strong convergence order equal to $1$:
\begin{align*}
\sup_{n\le N}\Big\|\|X^N(t_n)-X(t_n)\|_{L^{q}}\Big\|_{L^p(\Omega)}
\le C(X_0,Q,T,p,q)\delta t.
\end{align*}
This strong convergence result is restricted to dimension $d=1$, due to a loss of exponential integrability of the auxiliary process $Z^N$ in higher dimension. Further study is required to overcome this issue.

This article is organized as follows. Some preliminaries are given in Section \ref{sec-2}.
The variational approach to deal with the case of space-time white noise  and  $Q$-Wiener processes,
as well as some properties of the auxiliary problem,
and one main strong convergence rate result, are 
given in Section \ref{sec-3}.
In Section \ref{sec-4},  stability and exponential integrability properties, of the exact solution and of a new auxiliary processes, are studied. Finally, we establish the optimal strong convergence rate of this proposed scheme in dimension 1.

We use $C$ to denote a generic constant, independent of the time step size $\delta t$, which differs from one place to another.

\section{Preliminaries}\label{sec-2}

In this section, we first introduce some useful notations
and further assumptions. Let $T>0$, $\delta t$ is the time step size, $N$ is the positive integer such that $N \delta t=T$, and let $\{t_k\}_{k\le N}$ be the grid points, defined by $t_k=k\delta t$.
We denote by 
$\HH=L^2(\OO)$, $L^q=L^q(\OO)$, $1\le q<\infty$ and 
$\mathcal E=\mathcal C(\OO)$.
$A$ is the Dirichlet Laplacian operator, which generates 
an analytic and contraction $C_0$-semigroup $S(t), t\ge 0$ in $\HH$ and $L^q$.
It is well-known that the assumptions on $\OO$ implies that the existence of the eigensystem $\{\lambda_k, e_k\}_{k\in\N^+}$ of 
$\HH$, such that $\lambda_k>0$, $-Ae_k=\lambda_k e_k$ and $\lim\limits_{k\to \infty}\lambda_k=\infty$.
Let
 $\mathbb W^{r,q}$ is the Banach space equipped with the norm
 $\|\cdot \|_{\mathbb W^{r,q}}:= \| (-A)^{\frac {r}2}\cdot \|_{L^q}$ 
for  the fractional power $(-A)^{\frac{r}{2}}, r\ge 0$.
The identities $\HH^1=H^1_0$ and $\HH^2=H^1_0\cap H^2$ are frequently  used in Section \ref{sec-4}.

Given two separable Hilbert spaces $\mathcal H$ and $ \widetilde  H$, we denote by $\LL^0_2(\mathcal H,\widetilde H)$ the space 
of Hilbert-Schmidt operators from $\mathcal H$ into $\widetilde H$, equipped with the usual norm given by  $\|Q\|_{\LL^0_2(\mathcal H,\widetilde H)}=(\sum_{k\in \N^+}\|Qe_k\|^2_{\widetilde H})^{\frac{1}{2}}$, where $\N^+=\{1,2,\cdots\}$, and  the result does not depend on the orthonormal basis $\{e_k\}_{k\in\N^+}$ of $\mathcal H$. We denote by $\LL^{\bs}_2:=\LL^0_2(\HH,\HH^{\bs})$, for $\bs \in \N$.

Given a Banach space $E$, we denote by $R(\widetilde H, E)$ the space of $\gamma$-radonifying operators endowed with the norm defined by
$\|Q\|_{\gamma(\widetilde H, E)}=(\widetilde \E\|\sum_{k\in\N^+}\gamma_k Qe_k \|^2_{E})^{\frac 12}$,
where $(\gamma_k)_{k\in\N^+}$ is a sequence of independent $\mathcal N(0,1)$-random variables on a
probability space $(\widetilde \Omega,\widetilde \FFF, \widetilde \PP)$.
We also need the Burkerholder inequality in martingale-type 2 Banach spaces $E=L^q, q\in[2,\infty)$, (see e.g. \cite{Brz97,VVW08}): for some $C_{p,E}\in(0,\infty)$,
\begin{equation}\label{Burk}
\begin{aligned}
\Big\|\sup_{t\in [0,T]}\Big\|\int_0^t \phi(r)dW(r)\Big\|_{E}\Big\|_{L^p(\Omega)}
&\le 
C_{p,E}\|\phi\|_{\LL^p(\Omega;L^2([0,T]; \gamma(\widetilde H;E)))}\\
&=
C_{p,E}\Big(\E\Big(\int_0^T\|\phi(t)\|^2_{\gamma(\widetilde H;E)}dt\Big)^{\frac p2}\Big)^{\frac 1p}
\end{aligned}
\end{equation}
 and the following property (see \cite{VVW08}): for some $C_q\in(0,\infty)$,
 \begin{align}\label{Rad}
 \|\phi \|_{\gamma(\widetilde H, L^q)}^2
\le C_q
 \Big\|\sum_{k\in \N^+} (\phi e_k)^2\Big\|_{L^{\frac q2}},
 \quad \phi \in \gamma(\widetilde H, L^q). 
 \end{align}
The process $W:=\sum_{k\in\N^+} \beta_ke_k$ is the $\HH$-valued cylindrical Wiener process, where $(\beta_k)_{k\in\N^+}$ are independent Brownian motions defined on a filtered probability space $(\Omega, \FFF, (\FFF(t))_{t\in[0,T]},\PP)$. The driving noise is $W^Q:=\sum_{k}\beta_kQe_k$, where 
$Q$ is a bounded operator from $\HH$ to $E$. 
When $Q=I$, $E=\HH$, $W^Q$ is the standard cylindrical Wiener process, which corresponds to the case
of space-time white noise.  In Sections~\ref{sec-3} and~\ref{sec-4},  we will also consider more regular cases, with assumptions   $Q \in \LL_2^\bs$, $\bs \in \N$.

 The solution of the stochastic Allen-Cahn equation, Eq.~\eqref{ac}, is interpreted in a mild sense,
\begin{align}\label{exa}
X(t)=S(t)X_0+\int_{0}^tS(t-s)F(X(s))ds+\int_{0}^tS(t-s)dW^Q(s).
\end{align}
 Let $\omega(t)= \int_{0}^tS(t-s)dW^Q(s)$ be the so-called stochastic convolution. Then note that $Y(t)=X(t)-\omega(t)$ solves a random PDE (written in mild form):
\begin{align}\label{exaY}
Y(t)=S(t)X_0+\int_{0}^tS(t-s)F(Y(s)+\omega(s))ds.
\end{align}
We now introduce an auxiliary problem, and several auxiliary processes. The auxiliary problem is coming from writing the solution of the splitting scheme Eq.~\eqref{spl} as follows:
\begin{align*}
X_{n+1}&=S_{\delta t}\Phi_{\delta t}(X_n)+\int_{t_n}^{t_{n+1}}S(t_{n+1}-s)dW^Q(s)\\
&=S_{\delta t}X_n+\delta t S_{\delta t} \Psi_{\delta t}(X_n)
+\int_{t_n}^{t_{n+1}}S(t_{n+1}-s)dW^Q(s),
\end{align*}
where $\Psi_{\delta t}(z)=\frac {\Phi_{\delta t}(z)- z}{\delta t}, \Psi_0(z)=F(z)$, $\Phi_0(z)=z$.
Thus we get  for all $n\in\left\{0,\ldots,N\right\}$
\begin{align*}
X_{n}&=S({t_{n}})X_0+
\delta t\sum_{k=0}^{n-1} S({t_{n+1}-t_k})\Psi_{\delta t}(X_k)+\sum_{k=0}^{n-1}\int_{t_k}^{t_{k+1}}S(t_{n+1}-s)dW^Q(s).
\end{align*}
A continuous time interpolation, such that $X^N(t_n)=X_n$  for all $n\in\left\{0,\ldots,N\right\}$, is defined by 
\begin{align}\label{mil-spl}
X^N(t)
=S(t)X_0+\int_{0}^{t}S((t-\lfloor s\rfloor_{\delta t}))
\Psi_{\delta t}(X^N(\lfloor s\rfloor_{\delta t}))ds
+\int_{0}^tS(t-s)dW^Q(s),
\end{align}
where $\lfloor s\rfloor_{\delta t}=\max\{0,\delta t,2\delta t,\cdots\}\cap [0, s]$.

As observed in \cite{BG18}, the proposed splitting scheme can be viewed as  the exponential Euler method applied to the following auxiliary SPDE:
\begin{align}\label{aux}
dX^{\delta t}(t)=AX^{\delta t}(t)dt+\Psi_{\delta t}(X^{\delta t}(t))dt+dW^Q(t), \quad X^{\delta t}(0)=X_0. 
\end{align}
 The associated mild formulation is given by
\begin{align*}
X^{\delta t}(t)=S(t)X_0+\int_{0}^tS(t-s)\Psi_{\delta t}(X^{\delta t}(s))ds+
\int_{0}^tS(t-s)dW^Q(s).
\end{align*}
Let $Y^{\delta t}(t)=X^{\delta t}(t)-\omega(t)$, where $\omega$ is the stochastic convolution. Then $Y^{\delta t}$ is also solution of a random PDE:
\begin{align}\label{auxY}
Y^{\delta t}(t)=S(t)X_0+\int_{0}^tS(t-s)\Psi_{\delta t}(Y^{\delta t}(s)+\omega(s))ds.
\end{align}

We quote 
the following results from \cite{BG18}.
The estimates may be derived with elementary calculations. 

\begin{lm}\label{pre}
For every $\delta t_0\in (0,1)$ and $\delta t\in [0,\delta t_0)$, the mapping $\Phi_{\delta t}$ is globally Lipschitz
continuous, and the mapping $\Psi_{\delta t}$ is 
locally Lipschitz 
continuous and satisfies a one-side
Lipschitz  condition. More precisely, for $q=2m$, $m\in 
\N^+$,
\begin{align*}
|\Phi_{\delta t}(z_1)-\Phi_{\delta t}(z_2)|
&\le e^{C\delta t_0}|z_1-z_2|,\\
(\Psi_{\delta t}(z_1)-\Psi_{\delta t}(z_2))(z_1-z_2)^{q-1}
&\le e^{C\delta t_0}|z_1-z_2|^q,\\
|\Psi_{\delta t}(z_1)-\Psi_{\delta t}(z_2)|
&\le C(\delta t_0)|z_1-z_2|(1+|z_1|^2+|z_2|^2),\\
|\Psi_{\delta t}(z_1)-\Psi_{0}(z_1)|
&\le C(\delta t_0)\delta t(1+|z_1|^5).
\end{align*}
\end{lm}

\section{Strong convergence rate analysis of the splitting scheme approximating stochastic Allen-Cahn equation  by a variational approach}\label{sec-3}

 This section is devoted to the application of a variational approach, to derive strong convergence rates for the splitting scheme defined by Eq.~\eqref{spl}. The study includes the cases of the cylindrical Wiener process ($Q=I$, $q=1$) and of $L^q$-valued  $Q$-Wiener processes ($Q\in \gamma(\HH,L^q)$).

 We recall that in~\cite{BG18} it is proved that the scheme is convergent, when $d=1$ and $Q=I$. Precisely, assume that $X_0 \in \HH^{\beta_1}\cap \mathcal E$, for some $\beta_1>0$. Then
\begin{align*}
\lim_{\delta t \to 0}\E\Big[\sup_{n\le N}\|X^N(t_n)-X(t_n)\|^p\Big]=0.
\end{align*}
However, it is well-known that the standard approach used to derive strong rates of convergence using a Gronwall's inequality argument, cannot be applied, when the nonlinearity is not globally Lipschitz continuous. Additional properties, precisely giving exponential integrability for the exact and numerical solutions, are required.

Instead, in the present section, we overcome this issue using a variational approach, based on a different decomposition of the error introduced below.

 For convenience, throughout this article, we assume that $X_0$ is a deterministic function and that 
$\sup_{k\in \N^+}\|e_k\|_{\mathcal E} \le C$. The typical example to ensures that  $\sup_{k\in \N^+}\|e_k\|_{\mathcal E} \le C$   is the d dimensional cube $[0,1]^d$.

\subsection{A priori estimates and spatial regularity properties} 
We first deal with the case $Q=I$, $d=1$ and  recall the following well-known result about the stochastic convolution
(see e.g. \cite{DZ14}): for $2 \le q<\infty$,
\begin{align*}
\E\Big[\sup_{t\in[0,T]}\|\omega(t)\|_{L^q}^p\Big]\le C_p(T), \quad \E\Big[\sup_{t\in[0,T]}\|\omega(t)\|_{\mathcal E}^p\Big]\le C_p(T)<\infty.
\end{align*}

 The following lemma states standard a priori estimates for the processes $X$, $X^N$ and $Y^{\delta t}$ defined by Eq.~\eqref{ac},~\eqref{mil-spl} and~\eqref{auxY} respectively. For convenience, throughout this paper, we omit the mollification procedure to get the evolution of $\|\cdot\|_{L^q}$.

\begin{lm}\label{pri}
Let $d=1$, $Q=I$, $q=2m$, $m\in\N^+$, $p\ge 1$ and  $X_0\in L^q$. Then $X$,$Y^{\delta t}$ and $X^N$
satisfy 
\begin{align*}
\E\Big[\sup_{t\in [0,T]} \|X(t)\|_{L^q}^{p}\Big]<C(T,p,q)
(1+\|X_0\|_{L^q}^{p}),
\end{align*}
and
\begin{align*}
\E\Big[\sup_{t\in [0,T]} \|Y^{\delta t}(t)\|_{L^q}^{p}
\Big]+
\E\Big[\sup_{t\in [0,T]} \|X^{N}(t)\|_{L^q}^{p}
\Big]<C(T,p,q) (1+\|X_0\|_{L^q}^{p}).
\end{align*}
\end{lm}
\begin{proof}
For the a priori estimate for the exact solution $X$,
we refer to \cite{DZ14}.
Thus we focus on the a priori estimate of $Y^{\delta t}$
and $X^N$. The definition, Eq.~\ref{auxY}, of $Y^{\delta t}$ and the one-side Lipschitz condition on $\Psi_{\delta t}$ (see Lemma \ref{pre}), combined with H\"older and 
Young inequalities, imply that 
for $2\le q<\infty $, 
\begin{align*} 
\|Y^{\delta t}(t)\|_{L^q}^q&\le \|X_0\|_{L^q}^q
+
q \int_0^t \<A Y^{\delta t}(s), (Y^{\delta t}(s))^{q-2}Y^{\delta t}(s)\>ds\\
&\quad+q \int_0^t \<\Psi_{\delta t}(Y^{\delta t}(s)+\omega(s)),(Y^{\delta t}(s))^{q-2}Y^{\delta t}(s)\>ds\\
&\le 
\|X_0\|_{L^q}^q
+q \int_0^t \<\Psi_{\delta t}(Y^{\delta t}(s)+\omega(s))-\Psi_{\delta t}(\omega(s)),(Y^{\delta t}(s))^{q-2}Y^{\delta t}(s)\>ds\\
&\quad+q\int_0^t \<\Psi_{\delta t}(\omega(s)),(Y^{\delta t}(s))^{q-2}Y^{\delta t}(s)\>ds\\
&\le 
\|X_0\|_{L^q}^q
+C(\delta t_0,q) \int_0^t \|Y^{\delta t}(s)\|_{L^q}^qds\\
&\quad+
C(\delta t_0,q)\int_0^t\|\Psi_{\delta t}(\omega(s))\|_{L^q}\|Y^{\delta t}(s))\|_{L^q}^{q-1}ds\\
&\le 
\|X_0\|_{L^q}^q
+C(\delta t_0,q) \int_0^t \|Y^{\delta t}(s)\|_{L^q}^qds+
 C(\delta t_0,q)\int_0^t(1+\|\omega(s)\|_{L^{3q}}^{3q})ds.
\end{align*}
 Using the moment estimate on the stochastic convolution above, applying 
the Gronwall's inequality concludes the proof for $Y^{\delta t}$.

The estimate for $X^N$ is  proved using similar arguments. First, note that it is sufficient to control the values of $X^N$ at the grid points $t_n$, $n\le N$:

\begin{align*}
\E\Big[\sup_{t\in [0,T]} \|X^{N}(t)\|_{L^q}^{p}
\Big]\le C(p,q,T)\E\Big[\sup_{n\le N}\|X^N(t_{n})\|_{L^q}^p\Big].
\end{align*}

By the definition of $X^N(t_n)=X_n$, $n\le N$ and the Lipschitz continuity of $\Phi_{\delta t}$ stated in Lemma \ref{pre},  since $S(t)$ is a contraction semigroup, we obtain 
\begin{align*}
\|X^N(t_{n})-&\omega(t_n)\|_{L^q}
\le 
\Big\|S(\delta t)\Phi_{\delta t}(X^N(t_{n-1}))
-S(\delta t)\omega(t_{n-1})\Big\|_{L^q}\\
&\le 
\Big\|\Phi_{\delta t}(X^N(t_{n-1}))
-\Phi_{\delta t} (\omega(t_{n-1}))\Big\|_{L^q}
+\Big\|\Phi_{\delta t}(\omega(t_{n-1}))
- \omega(t_{n-1})\Big\|_{L^q}\\
&\le 
e^{C\delta t}\Big\|X^N(t_{n-1})
-\omega(t_{n-1})\Big\|_{L^q}
+C\delta t(1+\|\omega(t_{n-1})\|_{L^{3q}}^3).
\end{align*}
Then using
the discrete Gronwall's inequality, and the estimate on
the stochastic convolution, we get
\begin{align*}
\E\Big[\sup_{n\le N}\|X^N(t_{n})\|_{L^q}^p\Big]\le 
C(T,q,p),
\end{align*}
which concludes the proof 
of $X^N$.
\end{proof}

We now study spatial regularity properties of 
the processes $X^N$ and $X$.
 We first state a Lemma (see~\cite{DZ14}) concerning the factorization method.

\begin{lm}\label{frac}
Assume that $p>1$, $r\ge 0$, $\gamma>\frac 1p+r$ and
that $E_1$ and $E_2$ are Banach spaces such that
\begin{align*}
\|S(t)x\|_{E_1}\le Mt^{-r}\|x\|_{E_2}, \quad t\in[0,T], x\in E_2.
\end{align*}
Set
$
G_{\gamma}f(t):=\int_{0}^t(t-s)^{\gamma-1}S(t-s)f(s)ds, 
$
then, for $\gamma >\frac 1p+r$, one has
\begin{align*}
\|G_{\gamma} f\|_{C([0,T]; E_1)} \le C(M)\|f\|_{L^p(0,T; E_2)},
\end{align*}
if $f\in L^p([0,T];E_2)$.
\end{lm}

\begin{lm}\label{reg}
Assume that $d=1$, $Q=I$, $p\ge 2$ and  $\|X_0\|_{\HH^{\beta_1}}<\infty$, $\beta_1 >0$.
The solution $u$ satisfy the following estimate: if $\beta <\min(\frac 12,\beta_1)$, then
\begin{align*}
\E\Big[\sup_{0\le s\le T}\|X(t)\|_{\HH^{\beta}}^p\Big]\le C(p,T,\beta,X_0)<\infty.
\end{align*}
\end{lm}
\begin{proof}
It is known (see e.g. \cite{DZ14}) that, for $\beta<\frac{1}{2}$,
\begin{align*}
\E\Big[\sup_{0\le s\le T}\|\omega(s)\|_{\HH^{\beta}}^p\Big]
\le 
C(p,T).
\end{align*}
Thus we only need to study the regularity of 
$S(t)X_0$ and of the deterministic convolution 
$\int_0^tS(t-s)F(X(s))ds$.
First,
\begin{align*}
\|S(t)X_0\|_{\HH^{\beta_1}} \le \|X_0\|_{\HH^{\beta_1}}.
\end{align*}
For the deterministic convolution, by the Fubini theorem,
we have 
\begin{align*}
\int_0^tS(t-s)F(X(s))ds
&=\frac {\sin{\gamma\pi}}{\pi}\int_0^t(t-s)^{\gamma-1}S(t-s)Y_{\gamma}(s)ds,\\
Y_{\gamma}(t)&=\int_{0}^t(t-s)^{-\gamma }S(t-s)F(X(s))ds
\end{align*}
where we choose $\gamma<\frac 14$ such that the regularity result also holds for the stochastic convolution.
Notice that
$
\|S(t)x\|_{\HH^{\beta}}\le Mt^{-\frac \beta 2}\|x\|_{\HH}$, 
for  $\beta>0$.
Taking $E_1=\HH^{\beta}$ and $E_2=\HH$, $r=\frac \beta 2$.
Lemma \ref{frac} yields that for large enough $p$ and $\gamma> \frac \beta2+\frac 1p$,
\begin{align*}
&\E\Big[\sup_{0\le t\le T}\Big\|\int_0^tS(t-s)F(X(s))ds\Big\|_{\HH^{\beta}}^p\Big]\\
&\le 
C\E\Big[
 \int_0^T \Big\|Y_{\gamma}(t) \Big\|_{\HH}^pdt\Big]\\
&\le 
C\E\Big[
 \Big(\int_0^T t^{-2\gamma}\|S(t)\|_{\LL(\HH,\HH)}\Big(1+\sup_{r\in [0,T]}\|X(r)\|_{L^6}^3 \Big)dt\Big)^p\Big]\\
 &\le C(T,p,X_0).
\end{align*}
This concludes the proof.
\end{proof}

 Using standard arguments, including the use of a discrete Gronwall's lemma, and the two lemmas stated above, one may derive the following almost sure result (see~\cite{BG18} for similar arguments): assume $d=1$, $Q=I$, $\beta<\frac 12$, $X_0 \in \HH^{\beta_1}\cap \mathcal E$, $\beta_1>0$. Then almost surely, for some $C(\omega)\in(0,\infty)$, one has
\begin{align*}
\sup_{n\le N} \|X^N(t_n)- X(t_n)\| \le C(\omega)\delta t^{\min(\frac \beta 2, \frac {\beta_1} 2)}.
\end{align*}

We omit the details. As explained above, the variational approach used below allows us to go beyond this result and get a strong rate of convergence.

\subsection{Optimal  strong convergence rate in space-time white noise case}

We are now in position to apply the variational approach developed in~\cite{BJ16} in order to obtain strong convergence rates for the splitting scheme~\eqref{spl}.

We first state the main result of this section.
\begin{tm}\label{strong 1}
Assume that $d=1$, $Q=I$, $\|X_0\|_{L^{9q}}<\infty$,
$p\ge q=2m$, $m\in \N^+$ and $\eta<\frac 1q$.
Then  $X^N$  satisfies 
\begin{align*}
\Big\|\sup_{t\in[0,T]} \|X^N(t)- X(t)\|_{L^{q}}\Big\|_{L^p(\Omega)} \le C(T,X_0,p,q)\delta t^{\min(\frac 14, \eta)},
\end{align*} 
If in addition $\|X_0\|_{\mathbb W^{\beta,3q}}<\infty$, $\beta>0$,
then 
\begin{align*}
\Big\|\sup_{t\in[0,T]} \|X^N(t)- X(t)\|_{L^{q}}\Big\|_{L^p(\Omega)} \le C(T,X_0,p,q)\delta t^{\min(\frac 14, \frac \beta 2+\eta)},
\end{align*} 

\end{tm}

Note that, for $q\in[2,4)$, the first estimate of Theorem~\ref{strong 1} gives order of convergence $\frac{1}{4}$. If $q\in[4,\infty)$,the order of convergence $\frac{1}{4}$ is obtained thanks to the second estimate, under the assumption $\beta>\frac{1}{2}-\frac{2}{q}$.

Observe that the error can be decomposed as follows:
\begin{align*}
\|X^N(t)-X(t)\|_{L^q}&\le \|X^N(t)-\omega(t)-Y^{\delta t}(t)\|_{L^q}+\|Y^{\delta t}(t)+\omega(t)-X(t)\|_{L^q}\\
&\le \|X^N(t)-X^{\delta t}(t)\|_{L^q}+\|Y^{\delta t}(t)-Y(t)\|_{L^q}.
\end{align*}
Then Theorem~\ref{strong 1} is a straightforward consequence of the two auxiliary results stated below.

\begin{prop}\label{error2}
Assume that $d=1$, $Q=I$, $\|X_0\|_{L^{5q}}<\infty$. 
Then the proposed method $X^N$ is strongly convergent to
$X$ and satisfies 
\begin{align*}
\Big\|\sup_{t\in [0,T]} \|Y^{\delta t}(t)-Y(t)\|_{L^q}\Big\|_{L^p(\Omega)} \le C(T,X_0,p,q)\delta t,
\end{align*} 
where $p\ge q=2m$, $m\in \N^+$.
\end{prop}

Note that $Y^{\delta t}(t)-Y(t)=X^{\delta t}(t)-X(t)$, and in the case $q=2$, Proposition~\ref{error2} has already been proved in~\cite{BG18}.

\begin{prop}\label{error1}
Assume that $d=1$, $Q=I$, $\|X_0\|_{L^{9q}}<\infty$, $p\ge q=2m$, $m\in \N^+$. 
Then the proposed method $X^N$ satisfies for $\eta<\frac 1q$,
\begin{align}\label{main-err}
\Big\|\sup_{t\in[0,T]} \| X^N(t)-X^{\delta t}(t)\|_{L^{q}}\Big\|_{L^p(\Omega)} \le C(T,X_0,p,q)\delta t^{\min(\frac 1 4, \eta)}.
\end{align} 
If in addition assume that $\|X_0\|_{\mathbb W^{\beta,3q}}<\infty$, $\beta>0$, 
then  for $\eta<\frac 1q$,
\begin{align}\label{main-err1}
\Big\|\sup_{t\in[0,T]} \| X^N(t)-X^{\delta t}(t)\|_{L^{q}}\Big\|_{L^p(\Omega)} \le C(T,X_0,p,q)\delta t^{\min(\frac 14, \eta +\frac \beta 2)}.
\end{align}
\end{prop}

It now remains to prove Propositions~\ref{error2} and~\ref{error1}.

\begin{proof}[Proof of Proposition~\ref{error2}]
 Note that $Y^{\delta t}(0)=Y(0)$, and recall that $Y^{\delta t}(t)-Y(t)=X^{\delta t}(t)-X(t)$. Using the differential forms of the random PDEs~\eqref{exaY} and~\eqref{auxY},
\begin{align*}
&\|Y^{\delta t}(t)- Y(t)\|_{L^q}^q\\
&=q\int_{\epsilon}^t\<(Y^{\delta t}(s)-Y(s))^{q-2}(Y^{\delta t}(s)- Y(s)), 
AY^{\delta t}(s)-
A Y(s)\>ds\\
&\quad +q\int_{\epsilon}^t\<(X^{\delta t}(s)-X(s))^{q-2}(X^{\delta t}(s)- X(s)), 
\Psi_{\delta t}(X(s))-
\Psi_{\delta t}(X(s)))\>ds\\
&\quad+q\int_{\epsilon}^t\<(X^{\delta t}(s)-X(s))^{q-2}(X^{\delta t}(s)- X(s)), 
\Psi_{\delta t}(X(s)))-
F(X(s)))\>ds.
\end{align*}

Thanks to Lemma~\ref{pre}, combined with Young's inequality and Gronwall's lemma, we obtain  
\begin{align*}
&\|Y^{\delta t}(t)-Y(t)\|_{L^q}^q\\
&\le C(T)\delta t^{q}\int_0^T\Big(1+\Big\|Y(s)+\omega(s) \Big\|^{5q}_{L^{5q}}\Big)ds, 
\end{align*}
It remains to use the a priori estimates of Lemma~\ref{pri} to conclude the proof.
\end{proof}

To prove Proposition~\ref{error1}, we follow the approach from~\cite{BJ16}, and we introduce an additional auxiliary process,
\[
\widehat Y(t):=S(t)X_0+\int_{0}^tS(t-s)\Psi_{\delta t}(X^N(\lfloor s\rfloor_{\delta t}))ds
\]
for which the following auxiliary result is satisfied.

\begin{lm}\label{error0}
Assume that $d=1$, $Q=I$, $\|X_0\|_{L^{3q}}<\infty$, $p\ge q\ge 2$. Then for $0<\eta<1$, $s\ge \delta t$,
\begin{align}\label{lm-err}
\E\Big[\Big\|\widehat Y(s)+\omega(s)-X^N(\lfloor s \rfloor_{\delta t})\Big\|_{L^{q}}^p\Big]
\le C(T,p,\eta,X_0)(1+(\lfloor s \rfloor_{\delta t})^{-\eta p})(\delta t)^{\min(\frac 14,\eta)p}.
\end{align}
If assume in addition that $\|X_0\|_{\mathbb W^{\beta,q}}<\infty$, $\beta>0$, then  we have 
\begin{align}\label{lm-err1}
\E\Big[\Big\|\widehat Y(s)+\omega(s)-X^N(\lfloor s \rfloor_{\delta t})\Big\|_{L^{q}}^p\Big]
\le C(T,p,\eta,X_0)(1+(\lfloor s \rfloor_{\delta t})^{-\eta p})(\delta t)^{\min(\frac 14,\frac\beta 2+\eta)p}.
\end{align}
\end{lm}

\begin{proof}[Proof of Lemma~\ref{error0}]
By the definition of $\widehat Y$ and $X^N$, we get,
for $s\ge \delta t$,
\begin{align*}
&\Big\|\widehat Y(s)+\omega(s)-X^N(\lfloor s \rfloor_{\delta t})\Big\|_{L^{q}}\\
&=\Big\|S(s)X_0-S(\lfloor s \rfloor_{\delta t})X_0\Big\|_{L^{q}}\\
&\quad+\Big\|\int_{0}^sS(s-r)\Psi_{\delta t}(X^N(\lfloor s\rfloor_{\delta t}))dr-\int_{0}^{\lfloor s\rfloor_{\delta t}}S(\lfloor s\rfloor_{\delta t}-\lfloor r\rfloor_{\delta t})\Psi_{\delta t}(X^N(\lfloor s\rfloor_{\delta t}))dr
\Big\|_{L^{q}} \\
&\quad+\Big\|\int_0^sS(s-r)dW^Q(r)-\int_{0}^{\lfloor s\rfloor_{\delta t}}S(\lfloor s\rfloor_{\delta t}-r)dW^Q(r)
\Big\|_{L^{q}}\\
&:=I_1+I_2+I_3.
\end{align*}
Thanks to the smoothing properties of  the semigroup $S(t)$, we have for arbitrary $\eta<1$,
\begin{align*}
I_1&\le\Big\|S(\lfloor s \rfloor_{\delta t})(S(s-\lfloor s \rfloor_{\delta t})-I)X_0\Big\|_{L^{q}} \\
&\le C\Big\|A^{\eta}S(\lfloor s \rfloor_{\delta t})\Big\|_{\LL({L^{q}},{L^{q}})}
\Big\|A^{-\eta}(S(s-\lfloor s \rfloor_{\delta t})-I)X_0\Big\|_{L^{q}}\\
&\le C
(\lfloor s \rfloor_{\delta t})^{-\eta }
\delta t^{\eta }\|X_0\|_{L^{q}}.
\end{align*}

Then we turn to estimate the term $I_2$, for $s\ge \epsilon$,
\begin{align*}
I_2
&\le 
\Big\|\int_{0}^{{\lfloor s\rfloor_{\delta t}}}\Big(S(s-r)-S(\lfloor s\rfloor_{\delta t}-\lfloor r\rfloor_{\delta t})\Big)\Psi_{\delta t}(X^N(\lfloor s\rfloor_{\delta t}))dr
\Big\|_{L^{q}}\\
&\quad+\Big\|\int_{{\lfloor s\rfloor_{\delta t}}}^sS(s-r)\Psi_{\delta t}(X^N(\lfloor s\rfloor_{\delta t}))dr
\Big\|_{L^{q}}\\
&\le \int_{0}^{{\lfloor s\rfloor_{\delta t}}}
\Big\| S(s-r)(S(r-{\lfloor r\rfloor_{\delta t}})-I)\Psi_{\delta t}(X^N(\lfloor s\rfloor_{\delta t}))\Big\|_{L^{q}}dr\\
&\quad+\int_{0}^{{\lfloor s\rfloor_{\delta t}}}
\Big\| S({\lfloor s\rfloor_{\delta t}}-{\lfloor r\rfloor_{\delta t}})(S(s-{\lfloor s\rfloor_{\delta t}})-I)\Psi_{\delta t}(X^N(\lfloor s\rfloor_{\delta t}))\Big\|_{L^{q}}dr\\
&\quad+C\delta t \sup_{s\in[0,T]}\Big\|\Psi_{\delta t}(X^N(\lfloor s\rfloor_{\delta t}))\Big\|_{L^{q}}.
\end{align*}
Similar to the estimate of $I_1$, combing with Lemma \ref{pre}, we obtain  for $0<\eta<1$,
\begin{align*}
I_2&
\le 
C(T)(\delta t^{\eta}+\delta t)(1+\sup_{n\le N}\|X^N(t_n)\|_{L^{3q}}^3).
\end{align*}
Thanks to the Burkholder inequality, Eq.~\eqref{Burk}, from Section~\ref{pre},
\begin{align*}
\E[I_3^{p}]
&\le C(p)\E\Big[\Big\|\int_0^{\lfloor s\rfloor_{\delta t}}\bigl(S(s-r)-S(\lfloor s\rfloor_{\delta t}-r)\bigr)dW(r)\Big\|_{L^{q}}^p\Big]\\
&\quad+C(p)\E\Big[\Big\|\int_{\lfloor s\rfloor_{\delta t}}^{s}S(s-r)dW(r)\Big\|_{L^{q}}^p\Big]\\
&\le C(p)\Big(\int_0^{\lfloor s\rfloor_{\delta t}}\Big\|S(s-r)-S(\lfloor s\rfloor_{\delta t}-r)\Big\|_{\gamma(\HH,L^{q})}^2dr\Big)^{\frac p2}\\
&\quad+C(p)\Big(\int_{\lfloor s\rfloor_{\delta t}}^{s}\Big\|S(s-r)\Big\|_{\gamma(\HH,L^{q})}^2dr\Big)^{\frac p2}.
\end{align*}
Thanks to Eq.~\eqref{Rad},
 \begin{align*}
 \|\phi \|_{\gamma(\HH, L^{q})}^2
\le C_q
 \Big\|\sum_{k\in \N^+} (\phi e_k)^2\Big\|_{L^{\frac {q}2}},
 \quad \phi \in \gamma(\HH, L^{q}).
 \end{align*}

Recall that it is assumed that $\underset{k\in\N^+}\sup~\|e_k\|_{\mathcal E}\le C<\infty$. Moreover, one has the following useful inequality: for any $\gamma>0$ and any $\alpha\in[0,1]$,
\[
\underset{r\in(0,\infty)}\sup~ \sum_{k\in\N^+}r^{\frac{1}{2}+\alpha}\lambda_k^\alpha e^{-\gamma\lambda_k r}=C(\gamma,\alpha)<\infty.
\]
Using these properties,

 \begin{align*}
 \E[I_3^{p}]
&\le C\Big(\int_0^{\lfloor s\rfloor_{\delta t}}\sum_{k\in \N^+}\Big\|S(\lfloor s\rfloor_{\delta t}-r)\bigl(S(s-\lfloor s\rfloor_{\delta t})-I\bigr)e_k \Big\|_{L^{q}}^2dr\Big)^{\frac p2}\\
&\quad+C\Big(\int_{\lfloor s\rfloor_{\delta t}}^s\sum_{k\in\N^+}\|S(s-r)e_k\|_{L^{q}}^2dr\Big)^{\frac p2}\\
&\le C\Big(\int_0^{\lfloor s\rfloor_{\delta t}}\sum_{k\in \N^+}e^{-2\lambda_k (\lfloor s\rfloor_{\delta t}-r)}
(e^{-\lambda_k (\lfloor s\rfloor_{\delta t}-s)}-1)^2dr\Big)^{\frac p2}\\
&+
C\Big(\int_{\lfloor s\rfloor_{\delta t}}^s\sum_{k\in\N^+}e^{-2\lambda_k (s-r)}dr\Big)^{\frac p2}\\
&\le C\Big(\int_0^{\lfloor s\rfloor_{\delta t}}\sum_{k\in \N^+}e^{-2\lambda_k (\lfloor s\rfloor_{\delta t}-r)}
\lambda_k^{\frac{1}{2}} \delta t^{\frac{1}{2}} dr\Big)^{\frac p2}+C\Big(\int_{\lfloor s\rfloor_{\delta t}}^s (s-r)^{-\frac{1}{2}}dr\Big)^{\frac p2}\\
&\le C\Big(\int_0^{\lfloor s\rfloor_{\delta t}} r^{-\frac{3}{4}}dr \delta t^{\frac{1}{2}}\Big)^{\frac{p}{2}}+C\delta t^{\frac{p}{4}}\le C\delta t^{\frac{p}{4}}.
\end{align*}

Combining the estimates of $I_1$, $I_2$ and $I_3$, we obtain for $s\ge \delta t$,
\begin{align*}
\E\Big[\Big\|\widehat Y(s)+\omega(s)-X^N(\lfloor s \rfloor_{\delta t})\Big\|_{L^{q}}^p\Big]
&\le C(p)\Big(\E[I_1^p]+\E[I_2^p]+\E[I_3^p]\Big)\\
&\le C(T,p,q,X_0)\Big((\lfloor s \rfloor_{\delta t})^{-\eta p}
\delta t^{\eta p}+\delta t^{\frac p4}\Big)\\
&\le C(T,p,q,X_0)(1+(\lfloor s \rfloor_{\delta t})^{-\eta p})\delta t^{\min(\frac 14,\eta)p},
\end{align*}
which shows the first assertion. 

If in addition we have $\|X_0\|_{\mathbb W^{\beta,q}}<\infty$, $\beta>0$, alternatively we have
\begin{align*}
I_1&\le\Big\|A^{\eta}S(\lfloor s \rfloor_{\delta t})\|_{\LL(L^q,L^q)}
\|A^{-\eta}(S(s-\lfloor s \rfloor_{\delta t})-I)X_0\Big\|_{L^{q}} \\
&\le C(\lfloor s \rfloor_{\delta t})^{-\eta }\delta t^{\eta+\frac \beta 2}\|X_0\|_{\mathbb W^{\beta, q}},
\end{align*}
where $\eta+\frac \beta 2\le 1$, $0<\eta<1$.
Combing the previous estimation on $I_2$ and $I_3$,
this concludes the proof.
\end{proof}

It now remains to prove Proposition~\ref{error1}, using Lemma~\ref{error0}.
\begin{proof}[Proof of Proposition~\ref{error1}]
We first show the estimation \eqref{main-err} with the rough initial datum $X_0\in L^{9q}$.
Due to the definition of $X^N$ and $Y^{\delta t}$, we have 
\begin{align*}
&\|X^N(t)-\omega(t)-Y^{\delta t}(t)\|_{L^q}\\
&=\Big\|\int_{0}^tS(t-\lfloor s\rfloor_{\delta t})
\Psi_{\delta t}(X^N(\lfloor s\rfloor_{\delta t}))ds
-
\int_{0}^tS(t-s)\Psi_{\delta t}(Y^{\delta t}(s)+\omega(s))ds\Big\|_{L^q}\\
&\le \Big\|\int_{0}^tS(t-\lfloor s\rfloor_{\delta t})
\Psi_{\delta t}(X^N(\lfloor s\rfloor_{\delta t}))ds
-
\int_{0}^tS(t-s)\Psi_{\delta t}(X^N(\lfloor s\rfloor_{\delta t}))ds\Big\|_{L^q}\\
&\quad+\Big\|\int_{0}^tS(t-s)
\Psi_{\delta t}(X^N(\lfloor s\rfloor_{\delta t}))ds
-
\int_{0}^tS(t-s)\Psi_{\delta t}(Y^{\delta t}(s)+\omega(s))ds\Big\|_{L^q}.
\end{align*}
The first term is controlled by the smoothing properties 
of $S(t)$ and the uniformly boundedness of $\Psi_{\delta t}(X^N(\lfloor s \rfloor_{\delta t}))$.
For $0<\eta_1 <1$, we have
\begin{align*}
&\Big\|\int_{0}^tS(t-\lfloor s\rfloor_{\delta t})
\Psi_{\delta t}(X^N(\lfloor s\rfloor_{\delta t}))ds
-
\int_{0}^tS(t-s)\Psi_{\delta t}(X^N(\lfloor s\rfloor_{\delta t}))ds\Big\|_{L^q}\\
&=
\Big\|\int_{0}^tA^{\eta_1}S(t-s)
A^{-\eta_1}(S(s-\lfloor s\rfloor_{\delta t})-I)
\Psi_{\delta t}(X^N(\lfloor s\rfloor_{\delta t}))ds
\Big\|_{L^q}\\
&\le
\int_{0}^t C(t-s)^{-\eta_1}(s-\lfloor s\rfloor_{\delta t})^{\eta_1}\Big\|\Psi_{\delta t}(X^N(\lfloor s\rfloor_{\delta t}))\Big\|_{L^q}ds\\
&\le C(\eta)\delta t^{\eta_1}\int_{0}^t(1+\|X^N(\lfloor s\rfloor_{\delta t})\|^3_{L^{3q}})ds. 
\end{align*}
We use the auxiliary process $\widehat Y$ and \eqref{lm-err} in Lemma \ref{error0} to deal with the second term  since
$$\Big\|\int_{0}^tS(t-s)
\Psi_{\delta t}(X^N(\lfloor s\rfloor_{\delta t}))ds
-
\int_{0}^tS(t-s)\Psi_{\delta t}(Y^{\delta t}(s)+\omega(s))ds\Big\|_{L^q}=\|Y^{\delta t}(t)-\widehat Y(t)\|_{L^q}.$$
 By 
 the one-sided Lipschitz continuity of $\Psi_{\delta t}$
, H\"older and Young inequality, we have for  $ \delta t\le t$, 
\begin{align*}
&\|Y^{\delta t}(t)-\widehat Y(t)\|_{L^q}^q\\
&=\|Y^{\delta t}(\delta t)-\widehat Y(\delta t)\|_{L^q}^q+q\int_{\delta t}^t\<(Y^{\delta t}(s)-\widehat Y(s))^{q-2}(Y^{\delta t}(s)-\widehat Y(s)), 
AY^{\delta t}(s)-
A\widehat Y(s)\>ds\\
&\quad +q\int_{\delta t}^t\<(Y^{\delta t}(s)-\widehat Y(s))^{q-2}(Y^{\delta t}(s)-\widehat Y(s)), 
\Psi_{\delta t}(Y^{\delta t}(s)+\omega(s))-\Psi_{\delta t}(X^N(\lfloor s\rfloor_{\delta t}))\>ds\\
&\le 
\|Y^{\delta t}(\delta t)-\widehat Y(\delta t)\|_{L^q}^q
+q\int_{\delta t}^t\Big\<(Y^{\delta t}(s)-\widehat Y(s))^{q-2}(Y^{\delta t}(s)-\widehat Y(s)), \\
&\quad\quad\quad
\Psi_{\delta t}(Y^{\delta t}(s)+\omega(s))
-\Psi_{\delta t}(\widehat Y(s)+\omega(s))\Big\>ds\\
&\quad+q\int_{\delta t}^t\<(Y^{\delta t}(s)-\widehat Y(s))^{q-2}(Y^{\delta t}(s)-\widehat Y(s)),
\Psi_{\delta t}(\widehat Y(s)+\omega(s))-\Psi_{\delta t}(X^N(\lfloor s\rfloor_{\delta t}))\>ds\\
&\le
\|Y^{\delta t}(\delta t)-\widehat Y(\delta t)\|^q
+C(q)\int_{\delta t}^t\|Y^{\delta t}(s)-\widehat Y(s)\|_{L^q}^qds\\
&\quad+\int_{\delta t}^t\Big\| \Psi_{\delta t}(\widehat Y(s)+\omega(s))-\Psi_{\delta t}(X^N(\lfloor s\rfloor_{\delta t})) \Big\|^q_{L^q}ds.
\end{align*}
Then Gronwall's inequality yields that for $t\ge \delta t$, 
\begin{align*}
\|Y^{\delta t}(t)-\widehat Y(t)\|_{L^q}^q
&\le e^{CT}\|Y^{\delta t}(\delta t)-\widehat Y(\delta t)\|_{L^q}^q\\
&\quad+ e^{CT}\int_{\delta t}^T\|\Psi_{\delta t}(\widehat Y (s)+\omega(s))-\Psi_{\delta t}(X^N(\lfloor s\rfloor_{\delta t})) \|_{L^q}^qds.
\end{align*}
Using Lemma 2.1 and H\"older inequality,
\begin{align*}
\|\Psi_{\delta t}(z_1)-\Psi_{\delta t}(z_2)\|_{L^{q}}
\le 
C\|z_1-z_2\|_{L^{3q}}(1+\|z_1\|_{L^{3q}}^2+\|z_2\|_{L^{3q}}^2)
\end{align*}	
leads that 
\begin{align*}
\|Y^{\delta t}(t)-\widehat Y(t)\|_{L^q}
&\le C(T,q)\Bigg( \|Y^{\delta t}(\delta t)-\widehat Y(\delta t)\|_{L^q}
+ \sup_{s\in[0,T]}\Big(1+ \|\widehat Y(s)\|_{L^{3q}}^{2}+
\|\omega(s)\|_{L^{3q}}^{2}\\
&\quad+\|X^N(\lfloor s \rfloor_{\delta t})\|_{L^{3q}}^{2}\Big) \Big(\int_{\delta t}^T
 \Big\|\widehat Y(s)+\omega(s)-X^N(\lfloor s \rfloor_{\delta t})\Big\|_{L^{3q}}^q\Big)^{\frac 1q}  ds\Bigg).
\end{align*}
Since 
for $t\le \delta t$,
\begin{align*}
\sup_{t\in[0,\delta t]}\|Y^{\delta t}(t)-\widehat Y(t)\|_{L^q}
&\le C\Big\|\int_{0}^{t} S(t-s)\Psi_{\delta t}(Y^{\delta t}(s)+\omega(s))ds\Big\|_{L^q}\\
&\quad+
C\Big\|\int_{0}^t S(t-s)\Psi_{\delta t}(X^N(\lfloor s\rfloor_{\delta t}))ds\Big\|_{L^q}\\
&\le 
C\delta t\sup_{s\in[0,T]}\Big(1+\|Y^{\delta t}(s)\|_{L^{3q}}^{3}
+\|\omega (s)\|_{L^{3q}}^{3}+\|X^N(s)\|_{L^{3q}}^{3}\Big).
\end{align*} 
Taking expectation, together with  the above results and a priori estimate in Lemma \ref{pri}, H\"older inequality and  Minkowski's inequality, leads that, for $p\ge q$, $\eta<\frac 1q$,
\begin{align*}
&\E \Big[\sup_{t\in[\delta t,T]}\|Y^{\delta t}(t)-\widehat Y(t)\|_{L^q}^p\Big]\\
&\le C(T,\eta,p,q)\Bigg(
\E\Big[ \|Y^{\delta t}(\delta t)-\widehat Y(\delta t)\|_{L^q}^p\Big]+ \E\Bigg[\sup_{s\in[0,T]}\Big(1+ \|\widehat Y(s)\|_{L^{3q}}^{2p}+
\|\omega(s)\|_{L^{3q}}^{2p}\\
&\quad+\|X^N(\lfloor s \rfloor_{\delta t})\|_{L^{3q}}^{2p}\Big)\Big(\int_{\delta t}^T 
 \Big\|\widehat Y(s)+\omega(s)-X^N(\lfloor s \rfloor_{\delta t})\Big\|_{L^{3q}}^q  ds\Big)^{\frac pq}\Bigg]\Bigg)\\
 &\le 
 C\delta t\E\Big[\sup_{s\in[0,T]}\Big(1+\|Y^{\delta t}(s)\|_{L^{3q}}^{3q}
+\|\omega (s)\|_{L^{3q}}^{3q}+\|X^N(s)\|_{L^{3q}}^{3q}\Big)\Big]\\
&\quad+C\E\Bigg[\sup_{s\in[0,T]}\Big(1+ \|\widehat Y(s)\|_{L^{3q}}^{2p}+
\|\omega(s)\|_{L^{3q}}^{2p}+
\|X^N(\lfloor s \rfloor_{\delta t})\|_{L^{3q}}^{2p}\Big)\\
&\quad\quad\Big(\int_{\delta t}^T 
 \Big\|\widehat Y(s)+\omega(s)-X^N(\lfloor s \rfloor_{\delta t})\Big\|_{L^{3q}}^q  ds\Big)^{\frac pq}\Bigg]\\
&\le C\delta t\E\Big[\sup_{s\in[0,T]}\Big(1+\|Y^{\delta t}(s)\|_{L^{3q}}^{3p}
+\|\omega (s)\|_{L^{3q}}^{3p}+\|X^N(s)\|_{L^{3q}}^{3p}\Big)\Big]\\
&\quad+C
\Big\|\sup_{s\in[0,T]}\Big(1+ \|\widehat Y(s)\|_{L^{3q}}^{2p}+C
\|\omega(s)\|_{L^{3q}}^{2p}+
\|X^N(\lfloor s \rfloor_{\delta t})\|_{L^{3q}}^{2p}\Big\|_{L^{2}(\Omega)}\\
&\quad\times\Big\|\Big(\int_{\delta t}^T 
 \Big\|\widehat Y(s)+\omega(s)-X^N(\lfloor s \rfloor_{\delta t})\Big\|_{L^{3q}}^q  ds\Big)^{\frac 1q}\Big\|_{L^{2p}(\Omega)}^p.
\end{align*}
By Minkowski's inequality and \eqref{lm-err} in Lemma \ref{error0}, 
we get for $\eta <\frac 1q$,
\begin{align*}
&\Big\|\Big(\int_{\delta t}^T 
 \Big\|\widehat Y(s)+\omega(s)-X^N(\lfloor s \rfloor_{\delta t})\Big\|_{L^{3q}}^q  ds\Big)^{\frac 1q}\Big\|_{L^{2p}(\Omega)}\\
&\le \Big(\int_{\delta t}^T\Big(\E\Big[ \Big\|\widehat Y(s)+\omega(s)-X^N(\lfloor s \rfloor_{\delta t})\Big\|_{L^{3q}}^{2p} \Big]\Big)^{\frac q {2p}} ds\Big)^{\frac 1q}\\
&\le C\Big(1+\Big(\int_{\delta t}^T \lfloor s \rfloor_{\delta t}^{-\eta q} ds\Big)^{\frac 1q}\Big)\delta t^{\min(\frac 14, \eta)}\le C(T,p,q,\|X_0\|_{L^{9q}})\delta t^{\min(\frac 14, \eta)},
\end{align*}
which combing with the above estimation, yields that 
\begin{align*}
\E \Big[\sup_{t\in[\delta t,T]}\|Y^{\delta t}(t)-\widehat Y(t)\|_{L^q}^p\Big]\le C(T,p,q,\|X_0\|_{L^{9q}})\delta t^{\min(\frac 14, \eta)}.
\end{align*}
By the continuity of $Y^{\delta t}(t)$ and $\widehat Y(t)$, together the above estimations, 
we have for 
\begin{align*}
\E \Big[\sup_{t\in[0,T]}\|Y^{\delta t}(t)-\widehat Y(t)\|_{L^q}^p\Big]&\le 
\E \Big[\sup_{t\in[\delta t,T]}\|Y^{\delta t}(t)-\widehat Y(t)\|_{L^q}^p\Big]\\
&\quad+\E \Big[\sup_{t\in[0,\delta t]}\|Y^{\delta t}(t)-\widehat Y(t)|_{L^q}^p\Big]\\
&\le C(T,p,q,\|X_0\|_{L^{9q}})\delta t^{\min(\frac 14, \eta)},
\end{align*}
which establishes the first assertion \eqref{main-err}.
For the estimation  \eqref{main-err1}, we  use \eqref{lm-err1} to estimate 
the term $\E \Big[\sup_{t\in[0,T]}\|Y^{\delta t}(t)-\widehat Y(t)\|_{L^q}^p\Big]$ and the arguments are similar. 
\
\end{proof}

By this variational approach, we can deduce that  if $d=1$, $Q=I$, $p\ge q=2m$, $m\in \N^+$, $\beta>0$,
 $\eta<\frac 1q$, $X_0\in \mathbb W^{\beta,q}\cap \mathcal E$.
Then the strong convergence rate result  still holds, i.e., 
\begin{align*}
\Big\|\sup_{t\in [0,T]} \|X^N(t)- X(t)\|_{L^q}\Big\|_{L^p(\Omega)} 
\le 
C(T, p,q, X_0)\delta t^{\min(\frac 14, \frac {\beta}2+\eta)},
\end{align*} 
by using the estimation 
\begin{align*}
\|\Psi_{\delta t}(z_1)-\Psi_{\delta t}(z_2)\|_{L^{q}}
\le 
C\|z_1-z_2\|_{L^{q}}(1+\|z_1\|_{\mathcal E}^2+\|z_2\|_{\mathcal E}^2)
\end{align*}
and the procedures of Theorem \ref{strong 1}.
This above result gives the answer to the problem about  the strong convergence rates of splitting schemes appeared in \cite{BG18}.


\begin{rk}
This above variational approach, combining with some 
further analysis on the discrete stochastic convolution, may also be available for 
obtaining the optimal strong convergence rates of other splitting schemes, such as the splitting exponential Euler scheme and the splitting implicit Euler scheme in \cite{BG18}. This extension will be studied in future works.
\end{rk}

 To conclude this section, we give extensions of Theorem~\ref{strong 1}, when Eq.~\eqref{ac} is driven by a $Q$-Wiener process, in dimension $d\le 3$. We only sketch the proofs of the parts which require nontrivial modifications. Note that the order of convergence depends on the H\"older regularity exponents for the process $X$.


\begin{cor}\label{com}
Let $d\le 3$, $p\ge q=2m$, $m\in \N^+$, $\beta_1>0$,
 $\eta<\frac 1q$. Assume that $X_0\in \mathbb W^{\beta_1,q}\cap \mathcal E$ and that the   operators $A$ and $Q$ satisfy: $Ae_k=-\lambda_k e_k$, $Qq_k=\sqrt{q_k}e_k$, $q_k>0$, $k\in \N^+$, with eigenfunctions such that
$\|e_k\|_{\mathcal E}\le C,$ $\|\nabla e_k\|\le C\lambda_k^{\frac 12}$. Suppose that $\sum\limits_{k\in \N^+} q_k\lambda_k^{2\beta-1}<\infty$, for some $0<\beta<\frac 12$.
Then  we have
\begin{align*}
\Big\|\sup_{t\in [0,T]} \|X^N(t)- X(t)\|_{L^q}\Big\|_{L^p(\Omega)} 
\le 
C(T, p,q, X_0)\delta t^{\min( \beta , \frac {\beta_1}2+\eta)}.
\end{align*} 
\end{cor}

\begin{proof}
To prove that Lemma~\ref{pri} holds true, it is sufficient to check the estimate $\E[\sup\limits_{t\in[0,T]}\|\omega(t)\|^{p}_{\mathcal E}]\le C(T,p,Q)$ for $p\ge 1$. This is a consequence of~\cite[Theorem  5.25]{DZ14}.

It now remains to explain modifications concerning Lemma~\ref{error0}. More precisely, the control of the term $I_3$ is modified as follows:
\begin{align*}
\E[I_3^{p}]
&\le C(p)\E\Big[\Big\|\int_0^{\lfloor s\rfloor_{\delta t}}\bigl(S(s-r)-S(\lfloor s\rfloor_{\delta t}-r)\bigr)dW(r)\Big\|_{L^{q}}^p\Big]\\
&\quad+C(p)\E\Big[\Big\|\int_{\lfloor s\rfloor_{\delta t}}^{s}S(s-r)dW(r)\Big\|_{L^{q}}^p\Big]\\
&\le C\Big(\int_0^{\lfloor s\rfloor_{\delta t}}\sum_{k\in \N^+}(e^{-\lambda_k (s-r)} -e^{-\lambda_k (\lfloor s\rfloor_{\delta t}-r)})^2q_kdr\Big)^{\frac p2}+
C\Big(\int_{\lfloor s\rfloor_{\delta t}}^s\sum_{k\in\N^+}e^{-2\lambda_k (s-r)}q_kdr\Big)^{\frac p2}\\
&\le  C
\Big(\sum_{k\in\N^+}{\lambda_k^{-1}}(1-e^{-\lambda_k(s-\lfloor s\rfloor_{\delta t})})q_k\Big)^{\frac p2}
+C\Big(\sum_{k\in \N^+}\frac {q_k}{\lambda_k}(1-e^{-2\lambda_k(s-\lfloor s\rfloor_{\delta t})})\Big)^{\frac p2}\\
&\le 
C(\sum_{k\in \N^+}q_k\lambda_k^{2\beta-1})^{\frac p2}\delta t^{\beta p}.
\end{align*}
Applying the same techniques as above concludes the proof of Corollary~\ref{com}.
\end{proof}

\begin{cor}\label{gam}
 Let $d=1$, $\beta>0$ and $p\ge 2$. 
If $Q \in \LL_2^0$, $X_0\in \HH^{\beta}\cap \mathcal E$, then  there exists a constant $C=C(X_0,Q,T,p)$  such that
\begin{align*}
\Big\|\sup_{t \in [0,T]} \|X^N(t)- X(t)\|\Big\|_{L^p(\Omega)} 
\le 
C\delta t^{\frac 12}.
\end{align*} 
If $d=2,3$, $\|(-A)^{\frac 12}Q\|_{\LL^0_2}<\infty$,  $X_0\in \HH^{\beta}\cap\mathcal E$,   
then there exists a constant $C'=C'(X_0,Q,T,p)$  such that
\begin{align*}
\Big\|\sup_{t \in [0,T]} \|X^N(t)- X(t)\|\Big\|_{L^p(\Omega)} 
\le 
C'\delta t^{\frac 12}.
\end{align*}

\end{cor}
\begin{proof}
We first show the first assertion.
The assumptions ensures that the method to obtain the 
strong convergence rates of the splitting scheme in the 
case $Q=I$ is
also available for the case $Q\in \LL^0_2$.
We only need to show that the a priori estimate of $\omega$, and $I_3$ possess higher convergence speed than the case $Q=I$.
The Sobolev embedding theorem, the regularity result of stochastic convolution in \cite[Theorem 5.15]{DZ14} and Burkerholder inequality yield that for $p\ge 2$, there exsits $\frac 14<\beta<\frac 12$ such that
\begin{align*}
\E[\sup_{t\in[0,T]}\|\omega(t)\|_{\mathcal E}^p]\le \E[\sup_{t\in[0,T]}\|(-A)^{\beta}\omega(t)\|^p]\le C(Q,T,p),
\end{align*}
and 
\begin{align*}
\E[I_3^{p}]
&\le C(p)\E\Big[\Big\|\int_0^{\lfloor s\rfloor_{\delta t}}\bigl(S(s-r)-S(\lfloor s\rfloor_{\delta t}-r)\bigr)dW^Q(r)\Big\|^p\Big]\\
&\quad+C(p)\E\Big[\Big\|\int_{\lfloor s\rfloor_{\delta t}}^{s}S(s-r)dW^Q(r)\Big\|^p\Big]\\
&\le C(p)\E\Big[\Big(\int_0^{\lfloor s\rfloor_{\delta t}}
\Big\|(-A)^{-\frac 12}(S(s-\lfloor s\rfloor_{\delta t})-I)\Big\|^2\Big\|(-A)^{\frac 12}S(\lfloor s\rfloor_{\delta t}-r)Q\Big\|_{\LL^0_2}^2dr\Big)^{\frac p2}\Big]\\
&\quad+C(p)\E\Big[\Big(\int_{\lfloor s\rfloor_{\delta t}}^{s}\Big\|S(s-r)Q\Big\|^2dr\Big)^{\frac p2}\Big] \le C(Q)\delta t^{\frac p2}.
\end{align*}
The above properties, combined with the procedures in
the proof of Theorem \ref{strong 1} shows the first assertion.

Denote $W_{\gamma}=\int_{0}^t(t-s)^{-\gamma}S(t-s)(-A)^{\frac 12}dW^Q(s)$.
When $d=2,3$, $\|(-A)^{\frac 12}Q\|_{\LL^0_2}<\infty$,
Sobolev embedding theorem $\HH^{1+2\beta}\hookrightarrow \mathcal E$, $\frac 14<\beta<\frac 12$, together with the  the fractional method and Lemma \ref{frac}, yields that for $p> 2$, $\frac 14<\beta<\frac 12$, $\frac 12>\gamma>
\beta+\frac 1{p}$,
\begin{align*}
\E\Big[\sup_{s\in[0,T]}\|\omega(s)\|_{\mathbb E}^{p}\Big]
&\le \E\Big[\sup_{s\in [0,T]}\|\omega(s)\|_{\HH^{1+2\beta}}^{p}\Big]\le C \E\Big[\sup_{s\in [0,T]}\|G_{\gamma}W_{\gamma}(s)\|_{\HH^{2\beta}}^{p}\Big]\\
&\le C \int_0^T \E\Big[\|W_{\gamma}(s)\|^{p}_{\HH^{2\beta}}\Big]ds\\
&\le C\Big(\int_0^Ts^{-2\gamma}\|S(s)(-A)^{\frac 12}Q\|_{\LL_2^0}^2ds\Big)^{\frac p2}\le C(T,Q,p).
\end{align*}
Combining with the continuity of stochastic convolution
\begin{align*}
\E[I_3^{p}]
&\le C(p)\E\Big[\Big(\int_0^{\lfloor s\rfloor_{\delta t}}\Big\|(-A)^{-\frac 12}(S(s-\lfloor s\rfloor_{\delta t})-I)\Big\|^2 \Big\|(-A)^{\frac 12}S(\lfloor s\rfloor_{\delta t}-r)Q\Big\|_{\LL_2^0}^2dr\Big)^{\frac p2}\Big]\\
&\quad+C(p)\E\Big[\Big(\int_{\lfloor s\rfloor_{\delta t}}^{s}\Big\|S(s-r)Q\Big\|^2dr\Big)^{\frac p2}\Big] \le C(Q)\delta t^{\frac p2}.
\end{align*}
The a priori estimate of $\E[\sup\limits_{t\in[0,T]}\|\omega(t)\|_{\mathcal E}^p]$ and some procedures in the proof of Theorem \ref{strong 1}, we get the second assertion. 
\end{proof}

\section{ Higher strong convergence rate using exponential integrability properties (regular noise, dimension $1$)}\label{sec-4}

 This section is devoted to two contributions. First, we investigate exponential integrability properties of the exact and numerical solutions $X$ and $X^N$, in dimension $d=1,2,3$. We also derive useful a priori estimates in the $\mathbb{H}^2$ norm. This requires additional regularity conditions on the operator $Q$, and the initial condition $X_0$: it is assumed that $\|(-A)^{\frac 12}Q\|_{\LL^0_2}<\infty $ and $X_0 \in \HH^2$. Second, we prove that the splitting scheme, Eq.~\eqref{spl}, in the one-dimensional case $d=1$, has a strong order of convergence equal to $1$. Note that this higher order of convergence may be obtained since the stochastic convolution is not discretized. Up to our knowledge, this is the first proof that a temporal discretization scheme has a strong order of convergence equal to $1$ for the stochastic Allen-Cahn equation.

Like in Section~\ref{sec-3}, it is assumed for simplicity that the initial condition $X_0$ is deterministic. The extension of the results below to random $X_0$ is straightforward under appropriate assumptions: for instance, conditions of the type $\E[
e^{c\|X_0\|^2_{\HH^1} }]<\infty$ for some $c<\infty$ are required when studying exponential integrability properties.

\subsection{A priori estimates and exponential integrability properties}

In order to show  an improved strong error estimate,  with order $1$ in some cases, we need to prove additional a priori estimates, and to study the exponential integrability properties for $d=1,2,3$, in some well-chosen Banach spaces.

We first state the following result. The proof is standard, using It\^o's formula and the one-sided Lipschitz condition on $F$, and it is thus left to the interested readers.

\begin{lm}\label{reg-tra}
Assume that $d\le 3$, $\|(-A)^{\frac 12}Q\|_{\LL^0_2}<\infty $ and $X_0\in \HH^1$. Let $p\ge 1$. Then the solution $X$ of Eq.~\eqref{ac} satisfies the a priori estimates
\begin{align*}
\E \Big[\Big(\sup_{t\in[0,T]}\|X(t)\|^2+\int_{0}^{T}\|X(t)\|^2_{\HH^1}dt+\int_{0}^{T}
\|X(t)\|_{L^4}^4dt\Big)^p\Big]
\le C(X_0,T,Q,p)
\end{align*}
and 
\begin{align*}
\E \Big[\Big(\sup_{t\in [0,T]}\|X(t)\|^2_{\HH^1}+\int_{0}^{T}\|X(t)\|^2_{\HH^2}ds\Big)^p\Big]
\le C(X_0,T,Q,p).
\end{align*}
\end{lm}

To show the exponential integrability of $X$, we  quote an exponential  integrability lemma, see \cite[Lemma 3.1]{CHL16b}, see also~\cite{CHJ13} for similar results.

\begin{lm}\label{lm-exp}
Let $H$ be a Hilbert space, and let $X$ be an $H$-valued adapted stochastic process with continuous sample paths, satisfying $X_t=X_0+\int_0^t \mu(X_r)dr +\int_0^t \sigma(X_r)dW_r$ for all $t\in[0,T]$, where almost surely
$\int_0^T(\|\mu(X_t)\|+\|\sigma(X_t)\|^2)dt<\infty$.

Assume that there exist two functionals $\overline U$ and $ U \in \mathcal C^2(H;R)$, and  $\alpha\ge 0$, such that for all $t\in [0,T]$

\begin{align*}
&DU(x)\mu(x)
+\frac{{\rm tr}\big[D^2U(x)\sigma(x)\sigma^*(x)\big]}2 +\frac{\|\sigma^*(x) DU(x)\|^2}{2e^{\alpha t}}+\overline U(x)
\le \alpha U(x).
\end{align*}
Then 
\begin{align*}
\sup_{t\in [0,T]}\E \bigg[\exp\bigg( \frac {U(X_t)}{e^{\alpha t}}+\int_0^t\frac {\overline U(X_r)}{e^{\alpha r}}dr \bigg)\bigg]
\le e^{U(X_0)}.
\end{align*}
\end{lm}


 We are now in position to state a first exponential integrability result, which will be improved below. For the reader's convenience, we omit standard truncations and regularization procedures.

\begin{prop}\label{exp-prop}
Let $d\le 3$, and assume that $\|(-A)^{\frac 12}Q\|_{\LL^0_2}<\infty $ and $X_0\in \HH^1$. Then for any $\rho,\rho_1>0$,
there exist  $\alpha=\lambda(\rho,Q)\in (0,\infty)$ and $\alpha_1=\lambda(\rho_1,Q)\in (0,\infty)$, such that
\begin{align*}
&\E \Big[\exp\Big(e^{-\alpha t}\rho\|X(t)\|^2+2\rho\int_{0}^{t}e^{-\alpha s} \|\nabla X(s)\|^2ds+2\rho \int_{0}^{t}
e^{-\alpha s} \|X(s)\|_{L^4}^4ds\Big)\Big]
\le  e^{\rho\|X_0\|^2}.
\end{align*}
and
\begin{align*}
&\E \Big[\exp\Big(\Big(e^{-\alpha_1 t}\rho_1\|\nabla X(t)\|^2+2\rho_1\int_{0}^{t}e^{-\alpha_1 s} \|AX(s)\|^2ds\Big)\Big)\Big]\le e^{\rho_1\|\nabla X_0\|^2}.
\end{align*}
\end{prop}

\begin{proof} 
Define $\mu(x)=Ax-x^3+x, \sigma(x)=Q$,
$U(x)=\rho \|x\|^2$ and $U_1(x)=\rho_1 \|\nabla x\|^2$. 
Then note that for $\rho>0$,
\begin{align*}
&\<DU(x),\mu(x)\>+\frac 12\text{tr}[D^2U(x)\sigma(x)\sigma^*(x)]+\frac 12 \|\sigma(x)^*DU(x)\|^2\\
&=2\rho \<x,Ax-x^3+x\>
+\rho \|Q\|_{L^0_2}^2 +2\rho^2\|Q^*x\|^2\\
&\le -2\rho \|\nabla x\|^2+2\rho \|x\|^2-2\rho\|x\|_{L^4}^4
+\rho \|Q\|_{L^0_2}^2+2\rho^2\|x\|^2 \|Q\|_{L^0_2}^2\\
&\le -2\rho \|\nabla x\|^2-2\rho\|x\|_{L^4}^4
+\rho \|Q\|_{L^0_2}^2+(2\rho +2\rho^2\|Q\|_{L^0_2}^2)\|x\|^2.
\end{align*}
Let $\alpha\ge 2\rho +2\rho^2\|Q\|_{L^0_2}^2$,  and define
\[
\overline{U}(x)=2\rho \|\nabla x\|^2+2\rho \|x\|_{L^4}^{4}-\rho \|Q\|_{\mathcal{L}_2^0}^2.
\]
Then  one may apply Lemma~\ref{lm-exp}, which yields
\begin{align*}
&\E \Big[\exp\Big(e^{-\alpha t}\rho\|X(t)\|^2+2\rho\int_{0}^{t}e^{-\alpha s} \|\nabla X(s)\|^2ds+2\rho \int_{0}^{t}
e^{-\alpha s} \|X(s)\|_{L^4}^4ds\Big)\Big]\\
&\le \E \Big[e^{\frac {\rho\|Q\|^2_{\LL^0_2}}{\lambda }}e^{\rho \|X_0\|^2}\Big]\le e^{\rho \|X_0\|^2}.
\end{align*}
 The second inequality is obtained with similar arguments and the fact that $\HH^1=H^1_0$:
\begin{align*}
&\<DU_1(x),\mu(x)\>+\frac 12\text{tr}[D^2U_1(x)\sigma(x)\sigma^*(x)]+\frac 12 \|\sigma(x)^*DU_1(x)\|^2\\
&\le-2\rho_1 \<Ax,Ax\>-6\rho_1\<\nabla x,\nabla x  x^2\>
+\rho_1 \| \nabla Q\|_{\LL^0_2}^2 +(2\rho_1+2\rho_1^2\| \nabla Q\|_{\LL^0_2}^2)\|\nabla x\|^2.
\end{align*}
It remains to apply Lemma~\ref{lm-exp}, to get 
for $\alpha_1\ge 2\rho_1+2\rho_1^2\|\nabla Q\|_{\LL^0_2}^2$, 
\begin{align*}
&\E \Big[\exp\Big(\Big(e^{-\alpha_1 t}\rho_1\|X(t)\|^2+2\rho_1\int_{0}^{t}e^{-\alpha_1 s} \|AX(s)\|^2ds\Big)\Big)\Big]\\
&\le \E \Big[e^{\frac {\rho_1\|\nabla Q\|^2_{\LL^0_2}}{\alpha_1 }}e^{\rho_1\|\nabla X_0\|^2}\Big]\le e^{\rho_1\|\nabla X_0\|^2}.
\end{align*}

 This concludes the proof of Proposition~\ref{exp-prop}.
\end{proof}

The use of Gagliardo--Nirenberg--Sobolev inequalities (see e.g.~\cite{Nir59}) then allows us to improve the result of Proposition~\ref{exp-prop} as follows: we control exponential moments of the type $\E\Big[ \exp\Big(\int_0^Tc\|X(s)\|_{\mathcal E}^2ds\Big)\Big]$ with arbitrarily  large parameter $c\in(0,\infty)$. This result is crucial in the approach used below to obtain higher rates of convergence for the splitting scheme.

%

\begin{prop}\label{exp-tra}
Let $d\le 3$, and assume that $\|(-A)^{\frac 12}Q\|_{\LL^0_2}<\infty$ and $X_0\in \HH^1$. 
Then the solution $X$ of~\eqref{ac} satisfies, for any $c>0$,
\begin{align*}
\E\Big[ \exp\Big(\int_0^Tc\|X(s)\|_{\mathcal E}^2ds\Big)\Big]\le  C(c,d,T,X_0,Q)<\infty.
\end{align*}
\end{prop}

\begin{proof}
Assume first that $d=1$. Then we use the Gagliardo--Nirenberg--Sobolev inequality
$\|X\|_{\mathcal E} \le C_1\|\nabla X\|^{\frac 13}\|X\|^{\frac 23}_{L^4}$.

Thanks to Young's inequality, for all $\epsilon\in(0,1)$, there exists $C_1(\epsilon)\in(0,\infty)$ such that
\begin{align*}
\|X\|_{\mathcal E}^2
\le C_1\|\nabla X\|^{\frac 23}\|X\|^{\frac 43}_{L^4}
\le \Big(\epsilon \|\nabla X\|^{2}+ \epsilon \|X\|_{L^4}^{4}+C_1(\epsilon) \Big).
\end{align*}
Choose $\epsilon=\epsilon(c) \le \frac {\rho}{c e^{\alpha T}}\le 1$. Then, using Cauchy-Schwarz inequality, one gets
\begin{align*}
&\E\Big[ \exp\Big(\int_0^Tc\|X(s)\|_{\mathcal E}^2ds\Big)\Big]\\
&\le
\E\Big[ \exp\Big(\int_0^T  \epsilon c \|\nabla X\|^{2}+ \epsilon c \|X\|_{L^4}^{4}+C_1(\epsilon,c) \Big)ds\Big)\Big]\\
&\le e^{C_1(\epsilon,c)T}
\sqrt{\E\Big[ \exp\Big(\int_0^T  2\epsilon c \|\nabla X\|^{2}ds\Big)\Big]}
\sqrt{\E\Big[ \exp\Big(\int_0^T  2\epsilon c\|X\|_{L^4}^{4}ds\Big)\Big]}\\
&\le C(c,1,T,X_0,Q),
\end{align*}
 thanks to Proposition~\ref{exp-prop}, since $2\epsilon c\le \frac {\rho}{c e^{\alpha T}}$. This concludes the treament of the case $d=1$.

 When $d=2$, resp. $d=3$, we apply the Gagliardo-Nirenberg-Sobolev inequality, $\|X\|_{\mathcal E} \le C_2\|A X\|^{\frac 13}\|X\|^{\frac 23}_{L^4}$, resp. $\|X\|_{\mathcal E} \le C_3\|A X\|^{\frac 3 5}\|X\|^{\frac 25}_{L^4}$. In both cases, applying Young's inequality, for any $\epsilon\in(0,1)$, there exists $C_d(\epsilon)\in(0,\infty)$ such that
\begin{align*}
\|X\|_{\mathcal E}^2
\le \Big(\epsilon \|A X\|^{2}+ \epsilon \|X\|_{L^4}^{4}+C_d(\epsilon) \Big).
\end{align*}

%
Choose  $\epsilon=\epsilon(c) \le \min(\frac {\rho}{c e^{\alpha T}}, \frac {\rho_1}{ e^{\alpha_1 T}})\le 1$. Then
\begin{align*}
&\E\Big[ \exp\Big(\int_0^tc\|X(s)\|_{\mathcal E}^2ds\Big)\Big]\\
&\le
\E\Big[ \exp\Big(\int_0^T \epsilon c\|A X\|^{2}+ \epsilon c \|X\|_{L^4}^{4}+C_d(\epsilon,c) \Big)ds\Big)\Big]\\
&\le C(\epsilon,c,C_d,T)
\sqrt{\E\Big[ \exp\Big(\int_0^T 2\epsilon c \|AX\|^{2}ds\Big)\Big]}
\sqrt{\E\Big[ \exp\Big(\int_0^T 2\epsilon c\|X\|_{L^4}^{4}ds\Big)\Big]}\\
&\le C(c,C_d,T,X_0,Q),
\end{align*}
using Proposition~\ref{exp-prop}, and the condition on $\epsilon$.

This concludes the proof of Proposition~\ref{exp-tra}.
\end{proof}

 To conclude this section, we give an additional a priori estimate, with higher order spatial regularity for the solution $X$ of Eq.~\eqref{ac}.


\begin{prop}\label{hig-reg}
Let $d\le 3$, $\|(-A)^{\frac 12}Q\|_{\LL^0_2}<\infty$ and 
$X_0\in \HH^2$.
Then the solution $X\in \HH^2$, a.s.
Moreover for any $p\ge 2$,
\begin{align*}
\sup_{s\in[0,T]} \E\Big[\|X(s)\|_{\HH^2}^{p}\Big]\le C(T,Q,X_0,p).
\end{align*}
\end{prop}

\begin{proof}

By the mild form of $Y$,
we get 
\begin{align*}
\|Y(t)\|_{\HH^2} 
\le 
\|S(t)X_0\|_{\HH^2}+\Big\|\int_0^tS(t-s)F(Y+\omega(s))ds\Big\|_{\HH^2}.
\end{align*}
 The boundedness of $S(\cdot)$ and  the calculus inequality in the Sobolev spaces (see e.g. \cite{MB02}) 
 leads that
\begin{align*}
&\| S(t-s)F(Y(s)+\omega(s)))\|_{\HH^2}\\
&\le C\Big(\|Y(s)+\omega(s)\|_{\HH^2}
+\|X(s)\|_{\HH^2}\|X(s)\|^2_{\mathcal E}+
\|X(s)^2\|_{\HH^2}\|X(s)\|_{\mathcal E}\Big)\\
&\le C\Big(\|Y(s)+\omega(s)\|_{\HH^2}+\|X(s)\|_{\HH^2}\|X(s)\|^2_{\mathcal E}+
\|X(s)\|_{\HH^2}^2\|X(s)\|_{\mathcal E}\Big).
\end{align*}
Gronwall inequality, together with Sobolev embedding theorem, implies that
\begin{align*}
\|Y\|_{\HH^2} 
&\le 
C\exp(C\int_0^T\|X(s)\|_{\mathcal E}\|X(s)\|_{\HH^2}ds)\\
&\quad\Big(\sup_{t\in [0,T]} \|S(t)X_0\|+\int_0^T
(1+\|X(s)\|_{\mathcal E}\|X(s)\|_{\HH^2})\|\omega(t)\|_{\HH^2}dt\Big).
\end{align*}
Taking expectation, the exponential integrability in Proposition \ref{exp-tra}, and the regularity of 
the stochastic convolution,
\begin{align*}
\sup_{t\in[0,T]}\E\Big[\|A\omega(s)\|^p\Big]
&\le \sup_{t\in[0,T]}\E\Big[\|A\int_0^tS(t-s)dW^Q(s)\|^p\Big]\\
&\le C\sup_{t\in[0,T]}\E\Big[\Big(\int_0^t\|(-A)^{\frac 12}S(t-s)(-A)^{\frac 12}Q\|_{\LL^0_2}^2ds\Big)^{\frac p2}\Big]
\le C(T,Q,p),
\end{align*}
yields that
\begin{align*}
\E[\|X(s)\|^p_{\HH^2} ]
&\le 
C\E[\|\omega(s)\|^p_{\HH^2}]+
C\Big(\E\Big[\exp(2pC\int_0^T\|X(s)\|_{\mathcal E}\|X(s)\|_{\HH^2}ds)\Big]\Big)^{\frac 12}\\
&\quad\times \Big(\sup_{t\in [0,T]}\|S(t)X_0\|_{\HH^2}^{2p}+\E\Big[\Big(\int_0^T\|X(s)\|_{\mathcal E}\|X(s)\|_{\HH^2}\|\omega(s)\|_{\HH^2}ds\Big)^{2p}\Big]\Big)^{\frac 12}.
\end{align*}
By Gagliardo--Nirenberg inequality in  $d=1,2,3,$  and Young inequality, we get 
\begin{align*}
\|X(s)\|_{\mathcal E}\|X(s)\|_{\HH^2}
&\le C\|X(s)\|_{\HH^2}\|\nabla X(s)\|^{\frac 12}\|X(s)\|_{L^4}^{\frac 12}\\
&\le \epsilon\|X(s)\|_{\HH^2}^2+\epsilon \|\nabla X(s)\|^2
+\epsilon\|X(s)\|_{L^4}^4+C(\epsilon),\quad d=1\\ 
\|X(s)\|_{\mathcal E}\|X(s)\|_{\HH^2}
&\le C\|X(s)\|_{\HH^2}^{\frac 43}\|X(s)\|_{L^4}^{\frac 23}\\
&\le \epsilon\|X(s)\|_{\HH^2}^2+\epsilon\|X(s)\|_{L^4}^4+C(\epsilon),\quad d=2\\
\|X(s)\|_{\mathcal E}\|X(s)\|_{\HH^2}
&\le C\|X(s)\|_{\HH^2}^{\frac 85}\|X(s)\|_{L^4}^{\frac 25}\\
&\le \epsilon\|X(s)\|_{\HH^2}^2+\epsilon\|X(s)\|_{L^4}^4+C(\epsilon),\quad d=3.
\end{align*}
Combining with Proposition \ref{exp-prop}, we get the boundedness
of this exponential moment $\exp\Big(C\int_0^T\|X(s)\|_{\mathcal E}\|X(s)\|_{\HH^2}ds\Big)$.
The estimation of  $\E\Big[\Big(\int_0^T\|X(s)\|_{\mathcal E}\|X(s)\|_{\HH^2}\|\omega(s)\|_{\HH^2}ds\Big)^{2p}\Big]$ is similar.
Gagliardo--Nirenberg--Sobolev inequality, together with Sobolev embedding $L^4 \hookrightarrow \HH^1$, yields that 
for $d=1$,
\begin{align*}
&\E\Big[\Big(\int_0^T\|X(s)\|_{\mathcal E}\|X(s)\|_{\HH^2}\|\omega(s)\|_{\HH^2}ds\Big)^{2p}\Big] 
\le C \E\Big[(\int_0^T\|X(s)\|_{\HH^2}^{2}ds)^{2p}\Big]
\\
&\qquad\qquad\qquad\qquad\qquad\qquad\qquad\qquad+C \E\Big[\int_0^T\|X(s)\|_{\HH^1}^{4p}\|\omega(s)\|_{\HH^2}^{4p}ds\Big],\\
\text{for}\quad d=2,\\
&\E\Big[\Big(\int_0^T\|X(s)\|_{\mathcal E}\|X(s)\|_{\HH^2}\|\omega(s)\|_{\HH^2}ds\Big)^{2p}\Big] 
\le C \E\Big[(\int_0^T\|X(s)\|_{\HH^2}^{2}ds)^{2p}\Big]\\
&\qquad\qquad\qquad\qquad\qquad\qquad\qquad\qquad
+C \E\Big[\int_0^T\|X(s)\|_{\HH^1}^{4p}\|\omega(s)\|_{\HH^2}^{6p}ds\Big],\\
\text{for}\quad d=3,\\
&\E\Big[\Big(\int_0^T\|X(s)\|_{\mathcal E}\|X(s)\|_{\HH^2}\|\omega(s)\|_{\HH^2}ds\Big)^{2p}\Big] 
\le C \E\Big[(\int_0^T\|X(s)\|_{\HH^2}^{2}ds)^{2p}\Big]\\
&\qquad\qquad\qquad\qquad\qquad\qquad\qquad\qquad
+C \E\Big[\int_0^T\|X(s)\|_{\HH^1}^{4p}\|\omega(s)\|_{\HH^2}^{10p}ds\Big].
\end{align*}
Combining  the a priori estimate in Lemma \ref{reg-tra}
and the above inequalities, we finish the proof.\
\end{proof}

\subsection{Strong convergence order 1 of the splitting scheme}
In this part, we focus on the sharp strong convergence rate of $X^N$ in $d=1$.
The main reason why we could not obtain higher strong convergence rate in $d=2,3$ is that this splitting up strategy will destroy the exponential integrability in $L^4([0,T];L^4)$ and $L^2([0,T]; \HH^2)$  of the original equation and that the a priori estimate of the auxiliary process $Z^N$ in $\HH^2$ can not be obtained, since the Sobolev embedding $\mathcal E \hookrightarrow \HH^1$ does not hold.
We also remark the a priori estimate in Lemma \ref{reg-spl} holds for $d=2, 3$. 
The study of higher strong convergence  order for  numerical schemes in 
high dimensional case will
be studied in future works.

We first state the main result of this section. 
\begin{tm}\label{main}
Assume that $d= 1$, $\|(-A)^{\frac 12}Q\|_{\LL^0_2}<\infty$ and 
$X_0\in \HH^2$.
The proposed method possesses strong convergence order 1, i.e, for any $p\ge 1$,
\begin{align*}
\sup_{n\le N}\E\Big[\Big\|X(t_n)-X^N(t_n)\Big\|^p\Big]
\le C\delta t^{p}.
\end{align*}
\end{tm}

To obtain the higher strong order of the splitting scheme, we consider the following auxiliary predictable right continuous   process $Z^N$ such that $Z^N(t_n)=X^N(t_n)$, $n\le N$.
The process $Z^N$ is defined by recursion. Let $Z^N(0):=X_0$ and on each subinterval $[t_{n-1},t_{n}]$, $1\le n\le N$,
\begin{align*}
Z^N(t)&=\Phi_{t-t_{n-1}}(Z^N(t_{n-1})),\quad t\in [t_{n-1},t_n),\\
Z^N(t_n)&=S(\delta t)\Phi_{\delta t}(Z^N(t_{n-1}))+\int_{t_{n-1}}^{t_{n}}S(t_n-s)dW^Q(s).
\end{align*}
Since when $t\in [t_{n-1},t_n)$, 
\begin{align*}
dZ^N=F(Z^N(t))dt,
\end{align*}
we rewrite the definition of $Z^N$ into 
an integration form,
\begin{align}\label{def-spld}
Z^N(t)&=Z^N(t_{n-1})+\int_{t_{n-1}}^t F(Z^N(s))ds, \quad t\in [t_{n-1},t_n),\\\label{def-spls}
Z^N(t_n)&=
S(\delta t)Z^N(t_{n-1})+\int_{t_{n-1}}^{t_n} S(\delta t )F(Z^N(s))ds+\int_{t_{n-1}}^{t_{n}}S(t_n-s)dW^Q(s).
\end{align}
Letting $n$ be $n-1$ in the above equation and then plugging it
into Eq. \eqref{def-spld} yields that 
\begin{align*}
Z^N(t)&
= S(\delta t)Z^N(t_{n-2})+
\int_{t_{n-2}}^{t_{n-1}}S(\delta t)F(Z^N(s))ds+\int_{t_{n-1}}^t F(Z^N(s))ds\\
&\quad+\int_{t_{n-2}}^{t_{n-1}}S(t_{n-1}-s)dW^Q(s), \quad t\in [t_{n-1},t_n).
\end{align*}
Repeating this process, we get, for $t\in [t_{n-1},t_{n})$,
\begin{align*}
Z^N(t)&=S(t_{n-1})X_0+
\int_{0}^{t_{n-1}} S(t_{n-1}-\lfloor s \rfloor_{\delta t} )F(Z^N(s))ds\\
&\quad+\int_{t_{n-1}}^{t} F(Z^N(s))ds
+\int_{0}^{t_{n-1}} S(t_{n-1}-s)dW^Q(s),\\
\text{and}\\
Z^N(t_n)&=
S(t_{n})X(0)+
\int_{0}^{t_{n}} S (t_{n}-\lfloor s \rfloor_{\delta t} )F(Z^N(s))ds+\int_{0}^{t_{n}} S (t_{n}-s)dW^Q(s).
\end{align*}

\subsubsection{A priori estimate for the auxiliary process }
In order to get the strong convergence order, we also need the following a priori estimations  of $Z^N$.

\begin{lm}\label{reg-spl}
Assume that $d=1$, $\|(-A)^{\frac 12}Q\|_{\LL^0_2}<\infty $, ${\|X_0\|_{\HH^1}}<\infty$. 
Then  for $p\ge 2$,  the auxiliary process $Z^N$ satisfies
\begin{align*}
 \E\Big[\sup_{s\in [0,T]}\|Z^N(s)\|_{\HH^1}^{p}\Big]\le C(X_0,p,T,Q).
\end{align*}
\end{lm}

\begin{proof}
We first show the estimation of $\sup\limits_{s\in[0,T]}\E[\|Z^N(s)\|_{\HH^1}^p]\le C(T,p,Q,X_0)$.
Since similar arguments in Lemma 3.1 implies that 
$\sup\limits_{s\in[0,T]}\E[\|Z^N(s)\|^p]\le C(T,p,Q,X_0)$,
it sufficient to show $\sup\limits_{s\in[0,T]}\E[\|\nabla Z^N(s)\|^p]\le C(T,p,Q,X_0)$. For simplify the presentation, we only present the case $p=2$.
Consider the linear SPDE $d\widehat Z=A\widehat  Zdt+dW^Q(t)$ in local interval $[t_{n-1},t_n]$ with 
$\widehat Z(t_{n-1})=
\Phi_{\delta t}(Z^N(t_{n-1}))$, we have $\widehat Z(t_n)=Z^N(t_n)$.
By It\^o formula, we have 
\begin{align*}
\|\nabla Z^N(t_n)\|^2&=\|\nabla \Phi_{\delta t}(Z^N(t_{n-1}))\|^2
-2\int_{t_{n-1}}^{t_n}\<A\widehat Z,A\widehat Z\>ds\\
&\quad+2\int_{t_{n-1}}^{t_n}\<\nabla \widehat Z,\nabla dW(s)\>
+\int_{t_{n-1}}^{t_n}\|\nabla Q\|_{\LL^0_2}^2ds.
\end{align*}
Then taking expectation yields that
\begin{align*}
\E[\|\nabla Z^N(t_n)\|^2]
\le \E[\|\nabla \Phi_{\delta t}(Z^N(t_{n-1}))\|^2]
+\int_{t_{n-1}}^{t_n}\|\nabla Q\|_{\LL^0_2}^2ds.
\end{align*}
Since $\Phi_{t-t_{n-1}}Z^N(t_{n-1})$ is the solution of
$d\widetilde Z=F(\widetilde Z)dt$ with $\widetilde Z(t_{n-1})=Z^N(t_{n-1})$, the similar arguments 
yields that
\begin{align*}
\|\nabla \Phi_{t-t_{n-1}}(Z^N(t_{n-1}))\|^2 \le e^{C\delta t}\|\nabla Z^N(t_{n-1})\|^2.
\end{align*}
Combing the above estimations, 
we have for $t\in [t_{n-1},t_n)$,
\begin{align*}
\E\Big[\|\nabla Z^N(t)\|^2\Big]
&\le e^{C\delta t}\E\Big[\|\nabla Z^N(t_{n-1})\|^2\Big]\\
&\le e^{C\delta t}\Big(e^{C\delta t}\E\Big[\|\nabla Z^N(t_{n-2})\|^2\Big]
+C\delta t \Big)\\
&\le e^{CT}\|X_0\|^2+C(Q,T),
\end{align*}
which implies that $\sup\limits_{s\in[0,T]}\E[\|\nabla Z^N(s)\|^2]\le C(T,2,Q,X_0)$.
Similarly, we obtain the 
uniformly boundedness of $\sup\limits_{s\in [0,T]}\E\Big[\|Z^N(s)\|_{\HH^1}^{p}\Big]$, $p\ge 2$.

Now we are in position to show the desired result.
By the argument in Lemma 3.1, we 
have $\E[\sup\limits_{n\in N}\|X(t_n)\|_{L^q}^{p}]\le C$, $q=2m$. 
Then we aim to prove that $\E[\sup\limits_{n\in N}\|\nabla X(t_n)\|^{p}]\le C$. By the similar procedure of the previous proof of Lemma \ref{pri}, we get
\begin{align*}
\|\nabla \Big(X^N(t_{n})-\omega(t_{n})\Big)\|^2&\le (1+\delta t)\Big\|\nabla \Big(\Phi_{\delta t}(X^N(t_{n-1}))
-\Phi_{\delta t} (\omega(t_{n-1}))\Big)\Big\|^2\\
&\quad+C\delta t(1+\|\omega(t_{n-1})\|_{\HH^1}^6).
\end{align*}

Now, consider the SDEs $d\widetilde Z_{i}=F(\widetilde Z_{i})dt$ with different inputs $\widetilde Z_{1}(t_{n-1})=X^N(t_{n-1}))$ and $\widetilde Z_{2}(t_{n-1})=\omega(t_{n-1})$, we get $d(\widetilde Z_{1}-\widetilde Z_{2})=(F(\widetilde Z_{1})-F(\widetilde Z_{2}))dt$ for $t\in [t_{n-1},t_n]$.
Further calculations, together with Gagliardo--Nirenberg,
Holder and Young inequalities, yield that 
\begin{align*}
&\|\nabla\Big(\Phi_{t-t_{n-1}}(X^N(t_{n-1}))-\Phi_{t-t_{n-1} }(\omega(t_{n-1})) \Big)\|^2\\
&\le \|\nabla (X^N(t_{n-1})-\omega(t_{n-1}))\|^2
-\int_{t_{n-1}}^t\<(\widetilde Z_1-\widetilde Z_2)\nabla (\widetilde Z_1^2+\widetilde Z_1\widetilde Z_2+\widetilde Z_2^2),  \nabla \widetilde Z_1-\nabla \widetilde Z_2\>ds\\
&\le \|\nabla (X^N(t_{n-1})-\omega(t_{n-1}))\|^2
+C\int_{t_{n-1}}^t\|\nabla \widetilde Z_1-\nabla \widetilde Z_2\|^2 ds\\
&\quad+C\int_{t_{n-1}}^t\|\nabla (\widetilde Z_1^2+\widetilde Z_1\widetilde Z_2+\widetilde Z_2^2)\|^2\|(\widetilde Z_1-\widetilde Z_2)\|_{\mathcal E}^2ds\\
&\le \|\nabla (X^N(t_{n-1})-\omega(t_{n-1}))\|^2
+C\int_{t_{n-1}}^t\|\nabla \widetilde Z_1-\nabla \widetilde Z_2\|^2 ds\\
&\quad+C\int_{t_{n-1}}^t(\|\widetilde Z_1\|_{\HH^1}^4+\|\widetilde Z_2\|_{\HH^1}^4)\|\widetilde Z_1-\widetilde Z_2\|^2ds.
\end{align*}
On the other hand, the monotonicity of $F$ yields that 
the solution of 
$d\widetilde Z=F(\widetilde Z)dt$ satisfies for $t\in [t_{n-1},t_n]$,
\begin{align*}
\sup_{t\in [t_{n-1},t_n]}\|\widetilde Z(t)\|_{\HH^1}^2\le e^{C\delta t}(1+\|\widetilde Z(t_{n-1})\|_{\HH^1}^2).
\end{align*}
The above inequality yields that
\begin{align*}
&\|\nabla\Big(\Phi_{t-t_{n-1}}(X^N(t_{n-1}))-\Phi_{t-t_{n-1} }(\omega(t_{n-1})) \Big)\|^2\\
&\le e^{C\delta t}(\|\nabla (X^N(t_{n-1})-\omega(t_{n-1}))\|^2+e^{C\delta t}\delta t(1+\|X^N(t_{n-1})\|^6_{\HH^1}+\|\omega(t_{n-1})\|^6_{\HH^1}).
\end{align*}
Then discrete Gronwall's inequality leads that  
\begin{align*}
\|X^N(t_n)\|_{\HH^1}^2
\le C\|X_0\|_{\HH^1}^2+C\|\omega(t_n)\|^2_{\HH^1}+C\sum_{j=0}^{n-1}\delta t(1+\|X^N(t_j)\|_{\HH^1}^6+\|\omega(t_j)\|_{\HH^1}^6).
\end{align*}
Taking expectation, we obtain for any $p\ge 2$,
\begin{align*}
\E\Big[\sup_{n\le N}\|X^N(t_n)\|^p_{\HH^1}\Big]
\le C\Big(1+\|X_0\|_{\HH^1}^p+\sup_{n\le N}\E\Big[\|X^N(t_n)\|^{3p}_{\HH^1}\Big]
+\E\Big[\sup_{n\le N}\|\omega(t_n)\|_{\HH^1}^{3p}\Big]\Big).
\end{align*}
Denote $W_{\gamma}=\int_{0}^t(t-s)^{-\gamma}S(t-s)(-A)^{\frac 12}dW^Q(s)$.
By the fractional method and Lemma \ref{frac},
we have for $\beta<\frac 12$, $\frac 12>\gamma>
\beta+\frac 1{3p}$,
\begin{align*}
\E\Big[\sup_{s\in[0,T]}\|\omega(s)\|_{\HH^1}^{3p}\Big]
&\le \E\Big[\sup_{s\in [0,T]}\|\omega(s)\|_{\HH^{1+2\beta}}^{3p}\Big]\le C \E\Big[\sup_{s\in [0,T]}\|G_{\gamma}W_{\gamma}(s)\|_{\HH^{2\beta}}^{3p}\Big]\\
&\le C \int_0^T \E\Big[\|W_{\gamma}(s)\|^{3p}_{\HH^{2\beta}}\Big]ds\\
&\le C\Big(\int_0^Ts^{-2\gamma}\|S(s)(-A)^{\frac 12}Q\|_{\LL_2^0}^2ds\Big)^{\frac {3p}2}\le C(T,Q,p),
\end{align*}
which implies that $\E\Big[\sup\limits_{n\le N}\|X^N(t_n)\|^p_{\HH^1}\Big]\le C(T,X_0,p,Q).$
Then the definition of $Z^N$ yields that 
\begin{align*}
\E\Big[\sup_{s\in[0,T]}\|Z^N(s)\|^p_{\HH^1}\Big]
\le C\Big(1+\E\Big[\sup_{n\le N}\|X^N(t_n)\|_{\HH^1}^{3p}\Big]\Big)\le C(T,X_0,Q,p),
\end{align*} 
which completes the proof.
\end{proof}

Similar to the procedures in the proof of Proposition \ref{exp-tra}, we show  the following exponential integrability
of $Z^N$ which is the key to get the higher strong convergence rate.
The rigorous proof is that 
in each local interval, we first apply the truncated argument and the spectral Galerkin 
method, then use the It\^o formula and Fatou lemma
to get the evolution of Lyapunouv functions. 
For convenience, we omit these procedures.

\begin{prop}\label{reg-spl2}
Assume that $d=1$, $\|(-A)^{\frac 12}Q\|_{\LL^0_2}<\infty $, $\|X_0\|_{\HH^1}<\infty$. Then we have for any $c>0$,
\begin{align}\label{exp-aux}
\E\Big[\exp\Big(\int_0^Tc\|Z^N(s)\|_{\mathcal E}^2ds\Big)\Big]\le C(X_0,T,Q,c).
\end{align}
\end{prop}

\begin{proof}
In each subinterval $[t_{n-1},t_{n}]$, we define the process $\widehat Z$ as the solution of 
$d\widehat Z=A\widehat Zdt+dW^Q(t)$, with $\widehat Z(t_{n-1})=\Phi_{\delta t}Z^N(t_{n-1})$.
Denote $\mu(x)=Ax, \sigma(x)=Q$,
$U(x)=\rho \|x\|^2$ and $U_1(x)=\rho_1 \|\nabla x\|^2$. 
we get for $\rho,\rho_1>0$, 
\begin{align*}
&\<DU(x),\mu(x)\>+\frac 12\text{tr}[D^2U(x)\sigma(x)\sigma^*(x)]+\frac 12 \|\sigma(x)^*DU\|^2\\
&=2\rho \<x,Ax\>
+\rho \|Q\|_{L^0_2}^2 +2\rho^2\|Qx\|^2\\
&\le -2\rho \|\nabla x\|^2
+\rho \|Q\|_{L^0_2}^2+2\rho^2\|Q\|_{L^0_2}^2\|x\|^2.
\end{align*}
Lemma \ref{lm-exp} yields that for $\alpha \ge 2\rho^2\|Q\|_{L^0_2}^2$,
\begin{align*}
&\E\Big[\exp\Big(e^{-\alpha t_n}\rho\|Z^N(t_n)\|^2\Big)\Big]\le e^{C\delta t}\E\Big[\exp\Big(e^{-\alpha {t_{n-1}}}\rho\|\Phi_{\delta t}Z^N(t_{n-1})\|^2\Big)\Big].
\end{align*}
Since $\Phi_{t-t_{n-1}}Z^N(t_{n-1})$ is the solution of 
$d\widetilde Z=F(\widetilde Z)dt$ with $\widetilde Z(t_{n-1})=Z^N(t_{n-1})$ in $[t_{n-1},t_n]$, similar calculation, together with H\"older and Young inequality, yields
\begin{align*}
&\E\Big[\exp\Big(e^{-\alpha {t_{n-1}}}\rho\|\Phi_{\delta t}Z^N(t_{n-1})\|^2\Big)\Big]\\
&=\E\Big[\exp\Big(e^{-\alpha {t_{n-1}}}\rho\|Z^N(t_{n-1})\|^2-e^{-\alpha {t_{n-1}}}2\rho\int_{t_{n-1}}^{t_n} \|Z^N(s)\|_{L^4}^4ds\\
&\quad+e^{-\alpha{t_{n-1}}}2\rho\int_{t_{n-1}}^{t_n} \|Z^N(s)\|^2ds
\Big)\Big]\\
&\le e^{C\delta t}\E\Big[\exp\Big(e^{\alpha {t_{n-1}}}\rho\|Z^N(t_{n-1})\|^2-e^{-\alpha {t_{n-1}}}\rho\int_{t_{n-1}}^{t_n} \|Z^N(s)\|_{L^4}^4ds\Big)\Big]\\
&\le e^{C\delta t}\E\Big[\exp\Big(e^{-\alpha {t_{n-1}}}\rho\|Z^N(t_{n-1})\|^2\Big)\Big].
\end{align*}
Then repeating the above procedures, 
\begin{align*}
\E\Big[\exp\Big(e^{-\alpha t_n}\rho\|Z^N(t_n)\|^2\Big)\Big]
&\le e^{C\delta t} \E\Big[\exp\Big(e^{-\alpha {t_{n-1}}}\rho\|Z^N(t_{n-1})\|^2\Big)\Big]\\
&\le e^{Ct_n}e^{\rho\|X_0\|^2}.
\end{align*}  
For $t\in [t_{n-1},t_n)$, we similarly have 
\begin{align*}
&\E\Big[\exp\Big(e^{-\alpha {t}}\rho\|Z^N(t)\|^2
+\int_{0}^te^{-\alpha s}\rho\|Z^N(s)\|^4_{L^4}ds \Big) \Big]\\
&\le \E \Bigg[\E\Big[\exp\Big(e^{-\alpha {t}}\rho\|Z^N(t)\|^2
+\int_{t_{n-1}}^te^{-\alpha s}\rho\|Z^N(s)\|^4_{L^4}ds \Big)\Big| \FFF_{t_{n-1}}\Big]\\
&\quad \times \exp\Big(\int_{0}^{t_{n-1}}e^{-\alpha s}\rho\|Z^N(s)\|^4_{L^4}ds \Big)
\Bigg]\\
&\le e^{C\delta t} \E\Big[\exp\Big(e^{-\alpha {t_{n-1}}}\rho\|Z^N(t_{n-1})\|^2+ \int_{0}^{t_{n-1}}e^{-\alpha s}\rho\|Z^N(s)\|^4_{L^4}ds\Big)\Big]\\
&\le e^{Ct_n}e^{\rho\|X_0\|^2}.
\end{align*}
Next, we focus on the exponential integrability in $\HH^1$.  Since  
$d\widehat Z=A\widehat Zdt+dW^Q(t)$ in $[t_{n-1},t_n]$, with $\widehat Z(t_{n-1})=\Phi_{\delta t}Z^N(t_{n-1})$, for $\rho_1>0$, we  have
\begin{align*}
&\<DU_1(x),\mu(x)\>+\frac 12\text{tr}[D^2U_1(x)\sigma(x)\sigma^*(x)]+\frac 12 \|\sigma(x)^*DU_1(x)\|^2\\
&=-2\rho_1 \<Ax,Ax\>
+\rho_1 \|\nabla Q\|_{L^0_2}^2 +2\rho_1^2\|\nabla Q\|_{L^0_2}^2\|\nabla x\|^2,
\end{align*}
which yields that for $\alpha_1\ge {2\rho_1^2\|\nabla Q\|_{L^0_2}^2}$, 
\begin{align*}
&\E\Big[\exp\Big(e^{-\alpha_1 t_n}\rho_1\|\nabla Z^N(t_n)\|^2\Big)\Big]
\le e^{C\delta t}\E\Big[\exp\Big(e^{-\alpha_1 {t_{n-1}}}\rho_1\|\nabla \Phi_{\delta t} Z^N(t_{n-1})\|^2\Big)\Big].
\end{align*}
Then the fact that $\Phi_{t-t_{n-1}} Z^N(t_{n-1})$ is 
the solution of $d\widetilde Z=F(\widetilde Z)dt$ in $[t_{n-1},t_n]$, with $\widetilde Z(t_{n-1})=Z^N(t_{n-1})$,
yields that for $ \alpha_1\ge {2 \widetilde  \rho_1}$, $ \widetilde \rho_1=  e^{2\rho_1^2\|\nabla Q\|_{L^0_2}^2 T}  \rho_1$,
\begin{align*}
&\E\Big[\exp\Big(e^{-\alpha_1 t_{n-1}}  \rho_1\|\nabla \Phi_{\delta t}Z^N(t_{n-1})\|^2
+\int_{t_{n-1}}^{t_n} e^{-\alpha_1 s}\\
&\quad 2  \rho_1\<\nabla Z^N(s), (Z^N(s))^2\nabla Z^N(s) \> ds\Big)\Big]\\
&\le e^{C \delta t }\E\Big[\exp\Big(e^{-\alpha_1 t_{n-1}} \rho_1\|\nabla Z^N(t_{n-1})\|^2\Big)\Big]
\end{align*} 
Repeating the above procedures and taking
$\alpha_1\ge \max(2\rho_1^2\|\nabla Q\|_{L^0_2}^2,2e^{ 2\rho_1^2\|\nabla Q\|_{L^0_2}^2 T}\rho_1)$, we obtain 
\begin{align*}
&\sup_{t\in[0,T]}\E\Big[\exp\Big(e^{-\alpha_1 t}\rho_1\|\nabla Z^N(t)\|^2\Big)\Big]\le Ce^{\rho_1\|\nabla X_0\|^2}.
\end{align*}
Now, we are in position to show the desired result \eqref{exp-aux}.
Gagliardo--Nirenberg--Sobolev inequality
$\|Z^N\|_{\mathcal E} \le C_1\|\nabla Z^N\|^{\frac 13}\|Z^N\|^{\frac 23}_{L^4}$, together with H\"older and Young inequalities, implies that
\begin{align*}
&\E\Big[\exp\Big(\int_0^Tc\|Z^N(s)\|_{\mathcal E}^2ds\Big)\Big]\\
&\le
\E\Big[\exp\Big(\int_0^T\frac 12\epsilon_1\|\nabla Z^N(s)\|^2+
\frac 12\epsilon_2 \| Z^N(s)\|^4_{L^4}+C(\epsilon_1,\epsilon_2,c))ds\Big)\Big]\\
&\le C(T,\epsilon_1,\epsilon_2,c)
\sqrt{\E\Big[\exp\Big(\int_0^T\epsilon_1\|\nabla Z^N(s)\|^2ds\Big)\Big]}
\sqrt{\E\Big[\exp\Big(\int_0^T\epsilon_2
 \| Z^N(s)\|^4_{L^4}ds\Big)\Big]}.
\end{align*}
Choosing $\epsilon_2\le e^{-\alpha T}\rho$,
we have 
\begin{align*}
\sqrt{\E\Big[\exp\Big(\int_0^T\epsilon_2
 \| Z^N(s)\|^4_{L^4}ds\Big)\Big]}
 \le e^{CT}e^{\frac \rho 2\|X_0\|^2}.
\end{align*}
Taking 
$\epsilon_1 \le \frac {e^{-\alpha_1 T}\rho_1}T$,
together with Jensen inequality, yields that
\begin{align*}
\sqrt{\E\Big[\exp\Big(\int_0^T\epsilon_1\|\nabla Z^N(s)\|^2ds\Big)\Big]}
&\le 
\sup_{s\in[0,T]}\sqrt{\E\Big[\exp\Big(T\epsilon_1\|\nabla Z^N(s)\|^2\Big)\Big]}\\
&\le e^{CT}e^{\frac \rho 2\|\nabla X_0\|^2}.
\end{align*}
The above two estimations leads the desired result.

\end{proof}

\subsubsection{Strong convergence order 1 of the splitting scheme}
After establish the a priori estimates and  the exponential integrability of 
both the exact and numerical solutions, 
we are in position to give the other main result on
the strong convergence rate of the splitting scheme.

\begin{proof}[Proof of Theorem~\ref{main}]
The mild representation  of $X$ \eqref{exa} and $X^N$ \eqref{aux} yields that

\begin{align*}
\|X(t_{n})-X^N(t_{n})\|&\le 
\Big\|\sum_{j=0}^{n-1}\int_{t_j}^{t_{j+1}} 
S(t_n-s)(F(X(s))-F(Z^N(s))ds\Big\|\\
&\quad+
\Big\|\sum_{j=0}^{n-1}\int_{t_j}^{t_{j+1}} 
(S(t_n-s)-S(t_n-t_j))F(Z^N(s))ds\Big\|
&:=II_1+II_2.
\end{align*}

By the smoothing properties of $S(t)$, $\HH^1$ is an algebra and $|F(z)|\le C(1+|z|^3)$, for $0<\eta<1$, $II_2$ is treat as follows:
\begin{align*}
II_2
&= 
\Big\| \int_{0}^{t_n} (-A)^{\eta}S(t_n-s)
(-A)^{-\eta}(I-S(s-\lfloor s \rfloor_{\delta t}))F(Z^N(s))ds\Big\|
\\
&\le 
C\delta t^{\frac 12+\eta} \Big(1+\sup_{s\in[0,T]}\Big\|Z^N(s)\Big\|_{\HH^1}^3\Big)\int_{0}^{t_n} \Big\| (-A)^{\eta}S(t_n-s)\Big\| ds\\
&\le C\delta t^{\frac 12+\eta} \Big(1+\sup_{s\in[0,T]}\Big\|Z^N(s)\Big\|_{\HH^1}^3\Big).
\end{align*}
For convenience, we introduce the mapping G such that
$F(z_1)-F(z_2)=G(z_1,z_2)(z_1-z_2)$, $z_1,z_2\in R$,
where $G(z_1,z_2)=-(z_1^2+z_2^2+z_1z_2)+1$.
$II_1$ is decomposed as 
\begin{align*}
II_1
&\le
 \Big\|\sum_{j=0}^{n-1}\int_{t_j}^{t_{j+1}} 
S(t_n-s)G(X(s),Z^N(s))(X(t_j)-Z^N(t_j))ds\Big\|\\
&\quad
+
\Big\| \sum_{j=0}^{n-1}\int_{t_j}^{t_{j+1}} 
S(t_n-s)G(X(s),Z^N(s))(X(s)-X(t_j))ds\Big\|\\
&\quad
+\Big\| \sum_{j=0}^{n-1}\int_{t_j}^{t_{j+1}} 
S(t_n-s)G(X(s),Z^N(s))(Z^N(s)-Z^N(t_j))ds\Big\|\\
&:=II_{11}+II_{12}+II_{13}.
\end{align*}
Direct calculations, together with Sobolev embedding  and  Gagliardo--Nirenberg inequality, yields that
\begin{align*}
II_{11}\le 
C\sum_{j=0}^{n-1}\int_{t_j}^{t_{j+1}} 
\Big(\|X(s)\|_{\mathcal E}^2+\|Z^N(s)\|_{\mathcal E}^2+1\Big)ds\|X(t_j)-Z^N(t_j)\|.
\end{align*}
and
\begin{align*}
II_{13}&\le
2C\sum_{j=0}^{n-1}\int_{t_j}^{t_{j+1}}
\Big(1+\|X(s)\|_{L^6}^2+\|Z^N(s)\|_{L^6}^2\Big)\Big\| Z^N(s)-Z^N(t_j)\Big\|_{L^6} ds\\
&\le 2C\sum_{j=0}^{n-1}\int_{t_j}^{t_{j+1}}
\sup_{s\in[0,t_n]}\Big(1+\|X(s)\|_{L^6}^2+\|Z^N(s)\|_{L^6}^2\Big)\Big\| \int_{t_j}^s F(Z^N(r))dr\Big\|_{L^6} ds\\
&\le 2C\delta t
\sup_{s\in[0,t_n]}\Big(1+\|X(s)\|_{L^6}^4+\|Z^N(s)\|_{L^6}^4+\|Z^N(s)\|_{L^{18}}^6\Big)\\
&\le  2C\delta t
\sup_{s\in[0,t_n]}\Big(1+\|X(s)\|_{\HH^1}^4+\|Z^N(s)\|_{\HH^1}^6\Big).
\end{align*}
For $II_{12}$, we have
\begin{align*}
II_{12}
&\le \Big\|\sum_{j=0}^{n-1}\int_{t_j}^{t_{j+1}} 
S(t_n-s)G(X(t_j),Z^N(t_j))(X(s)-X(t_j))ds\Big\|\\
&\quad+ \Big\|\sum_{j=0}^{n-1}\int_{t_j}^{t_{j+1}} 
S(t_n-s)\Big(G(X(s),Z^N(s))-G(X(t_j),Z^N(s))\Big)(X(s)-X(t_j))ds\Big\|\\
&\quad+ \Big\|\sum_{j=0}^{n-1}\int_{t_j}^{t_{j+1}} 
S(t_n-s)\Big(G(X(t_j),Z^N(s))-G(X(t_j),Z^N(t_j))\Big)(X(s)-X(t_j))ds\Big\|\\
&:=II_{121}+II_{122}+II_{123}.
\end{align*}
Using the mild form of $X(s)$ \eqref{exa} and Sobolev embedding $\HH^1  \hookrightarrow\mathcal E$, we have 
\begin{align*}
III_{121}
&\le \sum_{j=0}^{n-1}\int_{t_j}^{t_{j+1}} 
\Big\|S(t_n-s)G(X(t_j),Z^N(t_j))\int_{t_j}^s S(s-r) F(X(r))dr\Big\|ds\\
&\quad
+\Big\|\sum_{j=0}^{n-1}\int_{t_j}^{t_{j+1}} 
S(t_n-s)G(X(t_j),Z^N(t_j))\int_{t_j}^s S(s-r) dW^Q(r)ds\Big\|\\
&\quad+C\sum_{j=0}^{n-1}\delta t^2(\|X(t_j)\|^2_{\mathcal E}+\|Z^N(t_j)\|^2_{\mathcal E})\|(-A)X(t_j)\|
\\
&\le 
C\delta t\sup_{s\in[0,t_n]}\Big(1+\|X(s)\|_{\HH^1}^5+\|Z^N(s)\|_{\HH^1}^5\Big)\\
&\quad+C\sum_{j=0}^{n-1}\delta t^2(\|X(t_j)\|^2_{\HH^1}+\|Z^N(t_j)\|^2_{\HH^1})\|(-A)X(t_j)\|\\
&\quad
+\Big\|\sum_{j=0}^{n-1}\int_{t_j}^{t_{j+1}} 
S(t_n-s)G(X(t_j),Z^N(t_j))\int_{t_j}^s S(s-r) dW^Q(r)ds\Big\|.
\end{align*}
For the last term, taking expectation, together with 
the independence of increments of Wiener process, the adaptivity of $X$, Fubini theorem and Burkholder-Davis-Gundy inequality,
yields that for $p\ge 2$, 
\begin{align*}
&\E\Big[\Big\|\sum_{j=0}^{n-1}\int_{t_j}^{t_{j+1}} 
S(t_n-s)G(X(t_j),Z^N(t_j))\int_{t_j}^s S(s-r) dW^Q(r)ds\Big\|^p\Big]\\
&=\E\Big[\Big\|\sum_{j=0}^{n-1}\int_{t_j}^{t_{j+1}}\int_{r}^{t_{j+1}} S(t_n-s)G(X(t_j),Z^N(t_j)) S(s-r) ds dW^Q(r)\Big\|^p\Big]\\
&\le 
C(p)\E\Big[\Big(\sum_{j=0}^{n-1}\int_{t_j}^{t_{j+1}}
\Big\| 
\int_{r}^{t_{j+1}} S(t_n-s)G(X(t_j),Z^N(t_j))S(s-r) ds Q\Big\|_{\LL^0_2}^2dr\Big)^{\frac p2}\Big]\\
&\le C(p)\delta t^{p}
\Big(\sum_{j=0}^{n-1}\int_{t_j}^{t_{j+1}}
\Big(\|X(t_j)\|^2_{L^p(\Omega;L^6)}+\|Z^N(t_j)\|^2_{L^p(\Omega;L^6)}+1\Big)\sum_{k\in\N^+}\Big\|Qe_k\Big\|_{\HH^1}^2ds\Big)^{\frac p2}\\
&\le C(T,Q,X_0,p)\delta t^{p}.
\end{align*}
The definition of $G$ implies that $G$ is symmetric  and  $|G(z_1,z_2)-G(z_1,z_3)|\le |z_1||z_2-z_3|+|z_2-z_3||z_2+z_3|$. Based on this property, we estimate  $III_{122}$ and $III_{123}$ as 
\begin{align*}
&III_{122}+III_{123}\\
&\le 2C\sum_{j=0}^{n-1}\int_{t_j}^{t_{j+1}}
\| X(s)-X(t_j)\|_{L^6}^2(\|X(s)\|_{L^6}+\|X(t_j)\|_{L^6}+\|Z^N(s)\|_{L^6})ds\\
&\quad+ 2C\sum_{j=0}^{n-1}\int_{t_j}^{t_{j+1}}
\| X(s)-X(t_j)\|_{L^6}\|Z^N(s)-Z^N(t_j)\|_{L^6}(\|Z^N(s)\|_{L^6}+\|Z^N(t_j)\|_{L^6}+\|X(t_j)\|_{L^6})ds.
\end{align*}
The continuity of $X$, the right continuity of $Z^N$ and 
Sobolev embedding theorem 
lead that for $s\in[t_{j},t_{j+1})$, $\eta<1$,
\begin{align*}
&\|X(s)-X(t_j)\|_{L^6}\\
&\le 
\|(S(s)-S(t_{j}))X(0)\|_{L^6}
+\|\int_{0}^sS(s-r)F(X(r))dr-\int_0^{t_j}S(t_j-r)F(X(r))dr \|_{L^6}\\
&\quad+\|\int_{0}^sS(s-r)dW^Q(r)-\int_0^{t_j}S(t_j-r)dW^Q(r)\|_{L^6}
\\
&\le  C\delta t^{\frac 12} \|X_0\|_{\HH^2}
+ \|\int_{0}^{t_j}(S(s-r)-S(t_j-r))F(X(r))dr\|_{L^6} 
+ \|\int_{t_j}^{s}S(s-r)F(X(r))dr\|_{L^6} \\
&\quad+\|\int_{0}^{t_j}(S(s-r)-S(t_j-r))dW^Q(r)\|_{L^6} 
+ \|\int_{t_j}^{s}S(s-r)dW^Q(r)\|_{L^6} \\
&\le C\delta t^{\min(\frac 1 2,\eta)}\sup_{r\in [0,T]}\Big(\|X_0\|_{\HH^2}+\|X(r)\|_{\HH^1}+\|X(r)\|_{\HH^1}^3\Big)+\Big\|\int_{t_j}^sS(s-r)dW^Q(r)\Big\|_{L^6}\\
&\quad+\|\int_{0}^{t_j}(S(s-r)-S(t_j-r))dW^Q(r)\|_{L^6},
\end{align*}
where the two stochastic convolution terms can be bounded by Sobolev embedding 
$L^6 \hookrightarrow \HH^1$ and similar estimations for $I_3$ in Corollary \ref{gam}, 
and
\begin{align*}
&\|Z^N(s)-Z^N(t_j)\|_{L^6}\\
&\le
\|\int_{t_{n-1}}^sF(Z^N(r))dr \|_{L^6}
\le C\delta t \sup_{r\in [0,T]}\Big(1+\|Z^N(r)\|_{\HH^1}+\|Z^N(r)\|_{\HH^1}^3\Big).
\end{align*}
The above estimations, together with Young and H\"older inequality, implies that 
\begin{align*}
&III_{122}+III_{123}\\
&\le 2C\delta t^{\min(1, 2\eta)} \sup_{s\in[0,t_n]}
\Big(1+\|X_0\|^4_{\HH^2}+\|X(s)\|_{\HH^1}^{12}+\|Z^N(s)\|_{\HH^1}^{12}\Big)\\
&\quad+ 2C\sup_{s\in[0,t_n]}\Big(\|Z^N(s)\|_{\HH^1}+\|X(s)\|_{\HH^1}\Big)\sum_{j=0}^{n-1}\int_{t_j}^{t_{j+1}}
\Bigg(\Big\|\int_{t_j}^sS(s-r)dW^Q(r)\Big\|_{\HH^1}^2\\
&\quad+\|\int_{0}^{t_j}(S(s-r)-S(t_j-r))dW^Q(r)\|_{\HH^1}^2
\Bigg)ds.
\end{align*}
 Since $\|X(t_{n})-Z^N(t_{n})\|\le II_{11}+II_{2}+II_{13}+II_{121}+II_{122}
+II_{123},$
the discrete Gronwall's inequality in \cite[Lemma 2.6]{CH17} yields that 
\begin{align*}
\|X(t_{n})-Z^N(t_{n})\|
&\le 
C\exp\Big(2\sum_{j=0}^{n-1}\int_{t_j}^{t_{j+1}} 
\|X(s)\|_{\mathcal E}^2+\|Z^N(s)\|_{\mathcal E}^2ds\Big)\\
&\quad\times\Big(II_{2}+II_{13}+II_{121}+II_{122}
+II_{123}
\Big).
\end{align*}

Then taking expectation, together with  H\"older inequality, the a priori estimates in Lemma \ref{reg-tra},  Propositions  \ref{hig-reg}  and \ref{reg-spl2}, the continuity of stochastic convolution in the proof of Corollary \ref{gam} and exponential integrability of $X$ and $Z^N$ in Propositions \ref{exp-tra} and  \ref{reg-spl2}, we obtain for $p\ge 1$, $\frac 12<\eta<1$,
\begin{align*}
&\sup_{n\le N}\E\Big[\|X(t_n)-X^N(t_n)\|^p\Big]\\
&\le 
C(p) \sqrt[\frac 14]{\E\Big[\exp(4p\int_{0}^T\|X(s)\|_{\mathcal E}^2)\Big]}\sqrt[\frac 14]{\E\Big[\exp(4p\int_{0}^T\|Z^N(s)\|_{\mathcal E}^2)\Big]}\\
&\quad
\Big(\sqrt{\E[II_{2}^{2p}]}+\sqrt{\E[II_{13}^{2p}]}+\sqrt{\E[II_{121}^{2p}]}+\sqrt{\E[(II_{122}+II_{123})^{2p}]}\Big)\\
&\le C\delta t^{(\frac 12+\eta)p} \sqrt{\E\Big[1+\sup_{s\in[0,T]}\Big\|Z^N(s)\Big\|_{\HH^1}^{6p}\Big]}
+C\delta t^p
\sqrt{\E\Big[\sup_{s\in[0,T]}\Big(1+\|X(s)\|_{\HH^1}^{8p}+\|Z^N(s)\|_{\HH^1}^{12p}\Big)\Big]}\\
&\quad+ C\delta t^p \sqrt{\E\Big[\sup_{s\in[0,T]}\Big(1+\|X(s)\|_{\HH^1}^{10p}+\|Z^N(s)\|_{\HH^1}^{10p}\Big)\Big]}\\
&\quad+C\delta t^p\sum_{j=0}^{N-1}\delta t\sqrt{\E\Big[\|(-A)X(t_j)\|^{2p}(\|X(t_j)\|_{\HH^1}^{4p}+\|Z^N(t_j)\|_{\HH^1}^{4p})\Big]}\\
&\quad+C\delta t^{\min(1,2\eta )p}
\sqrt{\E\Big[1+\|X_0\|_{\HH^2}^{8p}+\sup_{s\in [0,T]}\Big(\|X(s)\|_{\HH^1}^{24p}+\|Z^N(s)\|_{\HH^1}^{24p}\Big)\Big]}\\
&\le C(T,p,Q,X_0)\delta t^{p}\Big(1+\sum_{j=0}^{N-1}\delta t\sqrt[\frac 14]{\E\Big[\|(-A)X(t_j)\|^{4p}\Big]}
\sqrt[\frac 12]{\E\Big[\|X(t_j)\|_{\HH^1}^{8p}+\|Z^N(t_j)\|_{\HH^1}^{8p}\Big]}\Big)\\
&\le C(T,p,Q,X_0)\delta t^{p},
\end{align*}
which completes the proof.

\end{proof}
As a direct consequence of the Theorem \ref{main} above, we have the following stronger
error estimation.
\begin{cor}
Assume that $d= 1$, $\|(-A)^{\frac 12}Q\|_{\LL^0_2}<\infty$ and 
$X_0\in \HH^2$.
Then for any $p\ge 1$ and $0<\eta<1$,
\begin{align*}
\Big\|\sup_{n\le N} \|X(t_n)-X^N(t_n)\|\Big\|_{L^p(\Omega)}
\le C\delta t^{\eta }.
\end{align*}
\end{cor}
\begin{proof}
Since for any $q' \ge 1$, based on Theorem \ref{main}, we have 
\begin{align*}
\E\Big[\sup_{n\le N} \Big\|X(t_n)-X^N(t_n)\Big\|^q\Big]
&\le \sum_{n\le N}\E\Big[\Big\|X(t_n)-X^N(t_n)\Big\|^q\Big]\le C\delta t^{q-1}.
\end{align*}
We completes the  proof by taking  $1-\frac 1q'\ge \eta$ and $q'\ge p$.
\end{proof}

\bibliographystyle{plain}
\bibliography{bib}

\end{document}